\documentclass[10pt,halfline,a4paper]{ouparticle}

\usepackage{amssymb}
\usepackage{amsmath}
\usepackage{amsthm}

\usepackage{float}
\usepackage{algorithm, caption}

\usepackage{algpseudocode}
\algtext*{EndWhile}
\algtext*{EndIf}
\algtext*{EndFor}
\algtext*{EndElse}

\usepackage[shortlabels]{enumitem}
\usepackage{dsfont}
\usepackage{bbm} 
\usepackage{fixmath}
\usepackage{eucal}
\usepackage{commath} 
\usepackage[T1]{fontenc}
\usepackage{graphicx}
\usepackage{tikz}
\usetikzlibrary{automata, positioning, arrows}
\tikzset{
->, 
>=stealth, 
node distance=2.7cm, 
every state/.style={thick, fill=gray!10}, 
initial text=$ $, 
  center coordinate/.style={
    execute at end picture={
      \path ([rotate around={180:#1}]perpendicular cs: horizontal line through={#1},
                                  vertical line through={(current bounding box.east)})
            ([rotate around={180:#1}]perpendicular cs: horizontal line through={#1},
                                  vertical line through={(current bounding box.west)});}}}
\usepackage{pgfplots}

\usepackage{IEEEtrantools,stackengine} 
\pgfplotsset{compat=1.17} 

\usepackage{verse}
\usepackage[hidelinks]{hyperref} 

\makeatletter
\newsavebox\myboxA
\newsavebox\myboxB
\newlength\mylenA

\newcommand*\xoverline[2][0.75]{%
    \sbox{\myboxA}{$\m@th#2$}%
    \setbox\myboxB\null
    \ht\myboxB=\ht\myboxA%
    \dp\myboxB=\dp\myboxA%
    \wd\myboxB=#1\wd\myboxA
    \sbox\myboxB{$\m@th\overline{\copy\myboxB}$}
    \setlength\mylenA{\the\wd\myboxA}
    \addtolength\mylenA{-\the\wd\myboxB}%
    \ifdim\wd\myboxB<\wd\myboxA%
       \rlap{\hskip 0.5\mylenA\usebox\myboxB}{\usebox\myboxA}%
    \else
        \hskip -0.5\mylenA\rlap{\usebox\myboxA}{\hskip 0.5\mylenA\usebox\myboxB}%
    \fi}
\makeatother

\newtheorem{customtheorem}{Theorem}

\newtheorem{theorem}{Theorem}[section]
\newtheorem{definition}[theorem]{Definition}
\newtheorem{corollary}[theorem]{Corollary}
\newtheorem{lemma}[theorem]{Lemma}

\theoremstyle{definition}

\newtheorem{remark}[theorem]{Remark}

\newcommand{\barc}{\bar{c}}

\def\N{{\mathbb N}}
\def\Z{{\mathbb Z}}
\def\Q{{\mathbb Q}}

\def\F{{\mathbb F}}

\def\calD{{\mathcal D}}

\def\M{{\mathcal M}}

\def\L{{\mathcal L}}

\def\U{{\mathcal U}}
\def\V{{\mathcal V}}

\DeclareMathOperator{\ch}{char}
\DeclareMathOperator{\rev}{rev}

\DeclareMathOperator{\vardiv}{div}

\DeclareMathOperator{\val}{val}
\DeclareMathOperator{\supp}{supp}

\newcommand{\tinfty}{t^{1/p^{\infty}}}

\newcommand{\PPFp}{\mathbb{F}_p((t^{1/p^{\infty}}))}

\newcommand{\PRFq}{\mathbb{F}_q[[t^{1/p^{\infty}}]]}

\newcommand{\PPF}{\mathbb{F}((t^{1/p^{\infty}}))}
\newcommand{\PRF}{\mathbb{F}[[t^{1/p^{\infty}}]]}
\newcommand{\PRFbar}{\xoverline{\mathbb{F}}[[t^{1/p^{\infty}}]]}

\newcommand{\FpbarQ}{\xoverline{\mathbb{F}}_p((t^{\mathbb{Q}}))}
\newcommand{\FbarQ}{\xoverline{\mathbb{F}}((t^{\mathbb{Q}}))}

\newcommand{\FpQ}{\mathbb{F}_p((t^{\mathbb{Q}}))}

\newcommand{\Fp}{\mathbb{F}_p}
\newcommand{\perf}{^{1/p^\infty}}
\newcommand{\Fpbar}{\xoverline{\mathbb{F}}_p}
\newcommand{\Fbar}{\xoverline{\mathbb{F}}}

\newcommand{\X}{\mathbb{X}}
\newcommand{\Y}{\mathbb{Y}}
\newcommand{\W}{\mathbb{W}}

\newcommand{\Lring}{\L_\text{ring}}
\newcommand{\Lval}{\L_{\text{val}}}
\newcommand{\Log}{\L_{\text{og}}}

\newcommand{\acl}{\text{acl}}
\newcommand{\tp}{\text{tp}}

\newcommand\Pow{\textup{Pow}}

\newcommand\andbf{\textbf{ and }}
\newcommand\orbf{\textbf{ or }}
\newcommand\notbf{\textbf{ not }}

\newcommand\true{\textsc{true}}
\newcommand\false{\textsc{false}}
\newcommand\PostradixStates{\textsc{PostradixStates}}
\newcommand\RelevantStates{\textsc{RelevantStates}}

\newcommand\enumeratealg{\textsc{enumerate}}
\newcommand\decidealg{\textsc{decide}}
\newcommand\ALG{\textsc{alg}}

\newcommand\encode{\textsc{encode}}

\newcommand\maxram{\textsc{maximal\_ramification}}
\newcommand\maxexp{\textsc{maximal\_expansion}}

\newcommand\lincomb{\textsc{linear\_combination}}

\newcommand\IsRoot{\textsc{IsRoot}}
\newcommand\isrootalg{\textsc{is\_root}}

\newcommand\postradix{\textsc{postradix}}
\newcommand\reachablestates{\textsc{reachable\_states}}
\newcommand\relevantstates{\textsc{relevant\_states}}

\newcommand\mindfa{\textsc{min\_dfa}}

\newcommand\transition{\textsc{transition\_graph}}

\newcommand\subgraph{\textsc{subgraph}}

\newcommand\cycles{\textsc{cycles}}

\newcommand\wellformed{\textsc{well\_formed}}

\newcommand\wellord{\textsc{well\_ordered}}
\newcommand\listwellord{\textsc{list\_well\_ordered}}

\newcommand\saguaro{\textsc{rooted\_saguaro}}

\newcommand\REV{\textsc{rev}}

\newcommand\add{\textsc{add}}
\newcommand\equals{\textsc{equals}}

\newcommand\mult{\textsc{multiply}}
\newcommand\multdfa{\textsc{multiply\_dfa}}

\newcommand\power{\textsc{power}}

\newcommand\accpaths{\textsc{accepting\_paths}}

\newcommand\listdfao{\textsc{list\_dfao}}

\newcommand\countroots{\textsc{count\_roots}}

\newcommand\isin{\textsc{is\_in}}

\newcommand\addpol{\textsc{additive\_multiple}}

\newcommand\dfa{\textsc{equivalent\_dfa}}

\newcommand\plabeling{\textsc{$p$\_labeling}}

\algnewcommand{\LineComment}[1]{\State \(\triangleright\) #1}

\title{Decidability of positive characteristic tame Hahn fields in $\L_t$}
\author{Victor Lisinski\thanks{The author was funded by an EPSRC award at the University of Oxford, with additional support from the Royal Swedish Academy of Sciences and Corpus Christi College Oxford.}}

\date{}

\newcommand{\Addresses}{{
  \vspace*{3em}
  \footnotesize

\textsc{Mathematical Institute, Woodstock Road, Oxford OX2 6GG.}\par\nopagebreak
  \textit{E-mail address}: \texttt{lisinski@maths.ox.ac.uk}
}}

\begin{document}

\abstract{
We show that any positive characteristic tame Hahn field $\F((t^\Gamma))$ containing $t$ is decidable in $\L_t$, the language of valued fields with a constant symbol for $t$, if $\F$ and $\Gamma$ are decidable. In particular, we obtain decidability of $\PPFp$ and $\FpQ$ in $\L_t$. This uses a new AKE-principle for equal characteristic tame fields in $\L_t$, building on work by Kuhlmann, together with Kedlaya's work on finite automata and algebraic extensions of function fields. We also obtain an AKE-principle for tame fields in mixed characteristic.}

\maketitle

\tableofcontents

\section{Introduction}
While the first order theory, and in particular decidability, of the $p$-adic numbers $\Q_p$ is well understood thanks to the work by Ax and Kochen \cite{AxKo_II} and, independently, Ershov \cite{MR0207686}, a long standing open problem is that of decidability of the equal characteristic analogue $\F_p((t))$. Some recent progress have been made on this topic when restricting the question to existential decidability. In \cite{MR2018551}, Denef and Shoutens showed that the existential theory of $\F_p((t))$ in the language of rings with a constant symbol for $t$ is decidable assuming resolution of singularities in characteristic $p$. In \cite{AnFe16}, Anscombe and Fehm showed that the existential theory of $\F_p((t))$ is unconditionally decidable in the language of rings. The results by Anscombe and Fehm uses decidability results on tame fields in the one sorted language of valued fields established by Kuhlmann in \cite{Ku16} and the fact that finite extensions of $\F_p((t))$ inside $\FpQ$ are isomorphic to $\F_p((t))$ itself.

This important connection between the first order theories of $\F_p((t))$ and $\FpQ$ in the language of valued fields served as a motivation for looking more closely at the first order theory of $\FpQ$, and more generally of equal characteristic tame fields, in $\L_t$, the language of valued fields with a distinguished constant symbol for $t$. The main result of this paper is the following.

\begin{customtheorem}
\label{thm: main decidability result}
Let $\F$ be a perfect field of characteristic $p$ which is decidable in the language of rings and let $\Gamma$ be a $p$-divisible decidable ordered group which is decidable in the language of ordered groups with a distinguished constant symbol $1$. Then, $\F((t^\Gamma))$ is decidable in $\L_t$, the language of valued fields with a distinguished constant symbol for $t$.
\end{customtheorem}

The language $\L_t$ is needed to establish a complete analogue to decidability results of extensions of $\Q_p$ in the language of rings, since $\F_p[t]$ in $\F_p((t))$ is like $\Z$ in $\Q_p$. Further motivating this analogy, Kartas has recently showed that decidability results of equal characteristic fields in $\L_t$ can be transferred to decidability results for mixed characteristic fields in the language of valued fields \cite{konstantinos2020decidability}. Hence, the results in this paper can be used to obtain decidability results for tame fields in mixed characterstic, which was previously unknown and has seemed to be inaccessible working in the language of rings.

To show Theorem \ref{thm: main decidability result}, we obtain a new Ax-Kochen-Ershov principle for certain tame fields in $\L_t$, extending results by Kuhlmann. This principle shows that the type of $t$ is axiomatised by one variable positive existential sentences, in the following sense. 

\begin{theorem}
\label{thm: main ake result}
Let $(L,v)$ be a tame field containing $\F_p(t)$, with $v(t)>0$ and let $T$ be the $\L_t$-theory of $(L,v)$. Let $(K,v)$ be the relative algebraic closure of $\F_p(t)$ in $L$. Suppose that $(K,v)$ is algebraically maximal and that $vL/vK$ is torsion free. Let
$$S=\left\{\exists X \left(f(X)=0 \land v(X)\ge 0 \right) \ \mid \ f(X) \ \textup{monic in} \  F_p[t][X]\right\} \cap T.$$
Let $(F, w)$ be a tame field containing $\F_p(t)$ such that $(F,w) \models S$. Suppose that $vL\equiv wF$ in the language of ordered groups with a distinguished constant symbol $\pi$, interpreted as $v(t)$ and $w(t)$ respectively. Suppose that $Lv\equiv Fw$ in the language of rings. Then $(L,v)\equiv (F,w)$ in the language $\L_t$.
\end{theorem}

While this principle was obtained with equal characteristic $p$ in mind, we also show that the following modified version of the result holds in mixed characteristic. 

\begin{theorem}
\label{thm: mixed char ake result}
Let $(L,v)$ be a tame field of mixed characteristic $(0,p)$ and let $T$ be the $\Lval$-theory of $(L,v)$. Let $(K,v)$ be the relative algebraic closure of $\Q$ in $L$. Suppose that $(K,v)$ is algebraically maximal and that $vL/vK$ is torsion free. Let
$$S=\left\{\exists X \left(f(X)=0 \land v(X) \ge 0 \right) \ \mid \ f(X) \ \textup{monic in} \ \Z[X]\right\} \cap T.$$ Let $(F, w)$ be a tame field of mixed characteristic $(0,p)$ such that $(F,w) \models S$. Suppose that $vL\equiv wF$ in the language of ordered groups with a distinguished constant symbol $\pi$, interpreted as $v(p)$ and $w(p)$ respectively. Suppose that $Lv\equiv Fw$ in the language of rings. Then $(L,v)\equiv (F,w)$ in the language $\Lval$.
\end{theorem}

Theorem \ref{thm: main ake result} is combined with the work by Kedlaya on algebraic extensions of function fields \cite{MR2289431} and the approximation method by Lampert to obtain Theorem \ref{thm: main decidability result}.

\section{Preliminaries}
\subsection{Valued fields}
\label{sec: valued fields}
With a valuation on a field, we mean an additive valuation as defined for example in \cite{pre05}.

For a valued field $(K,v)$, i.e. a field $K$ with an associated valuation $v$, we will denote its value group by $vK$ and its residue field by $Kv$. We say that $(K,v)$ has equal characteristic if $\ch(K)=\ch(Kv)$, and that $(K,v)$ has mixed characteristic if $\ch(K)\neq\ch(Kv)$. For an element $\gamma\in vK$, a symbol $\bowtie \in\{<,\le, >,\ge\}$, and a subset $S$ of $K$, we write
$$S_{\bowtie\gamma}=\{x\in S \ | \ v(x)\bowtie\gamma\}.$$
In particular, the valuation ring of $K$ is written as $K_{\ge 0}$. For an element $a \in K_{\ge 0}$, we write $\bar{a} \in Kv$ for the projection of $a$ under the residue map.

For any $\gamma\in vK$, define the following sets
$$V_\gamma:=K_{>\gamma}$$
$$W_\gamma:=K_{\ge\gamma}.$$
Then $\{V_\gamma, W_\gamma \ | \ \gamma\in vK\}$ is a fundamental system of a Hausdorff topology on $K$, making it into topological ring. For details, see Theorem 20.16 in \cite{warner1989topological}. We call this topology the \textbf{valuation topology} on $(K,v)$. There is a unique topological field $\widehat{K}$ which is complete as a topological space and in which $K$ is dense under the valuation topology. There is also a unique valuation $\hat{v}$ on $\widehat{K}$ that extends $v$ and defines the topology on $\widehat{K}$. We will denote $\hat{v}$ by $v$ as well. For details, see Theorem 20.19 in \cite{warner1989topological}.

We will mainly be considering the $t$-adic valuation $v_t$ on fields consisting of formal expansions in $t$, such as $\F_p(t)$, $\F_p((t))$ and $\F_p((t^\Q))$ (see Section \ref{sec: hahn fields} for definitions). This valuation sends an element $\sum_\gamma a_\gamma t^\gamma$ to the minimal $\gamma_0$ such that $a_{\gamma_0}\neq 0$.

If $(L,v)$ is a valued field, we say that $(E,w)$ is a valued subfield of $(L,v)$ if $E$ is a subfield of $L$ and if $v|_E=w$. For a subfield $K$ of $L$, we will write $(K,v)$ to mean the valued subfield $(K,v|_K)$ of $(L,v)$.

Let $(K,v)$ be a valued field and let $L$ be a finite extension of $(K,v)$ of degree $n$. Then $v$ has finitely many extensions to $v$, call them $v_1,\ldots,v_s$. With $e_i:=(v_iL:vK)$ and $f_i:=[Lv_i:Kv]$, we have that $n$ satisfy the \textbf{fundamental inequality} 
$$n\ge\sum_{i=1}^s e_i f_i.$$

\begin{definition}
Let $(K,v)$ be a valued field. A finite extension $L/K$ of valued fields is \textbf{defectless} (with respect to $v$) if the fundamental inequality is an equality. We say that $(K,v)$ is \textbf{defectless} if any finite extensions of $(K,v)$ is defectless.
\end{definition}

\begin{definition}
Let $(K, v)$ be a valued field. A valued field extension $(L, v)$ of $(K, v)$ is an \textbf{immediate extension} if $[Lv : Kv] = 1$ and $(vL : vK) = 1$. If $(K,v)$ does not admit any proper (algebraic) immediate extension, then $(K,v)$ is said to be \textbf{(algebraically) maximal}.
\end{definition}

For a field $K$, we denote its algebraic closure by $\xoverline{K}$. The following standard result for approximating roots to polynomials over valued fields is found for example as Theorem 4.1.7 in \cite{pre05}.
\begin{theorem}[Krasner's Lemma]
Let $(K,v)$ be a valued field and let $x_0\in\xoverline{K}$ and suppose that its minimal poynomial
$$f(X)=\prod_{i=0}^n(X-x_i)\in K[X]$$
is separable. Let $y\in\xoverline{K}$ be such that $v(y-x_0)>\max_{i\neq 0}\{v(x_0-x_i)\}$. Then, $x_0\in K(y)$.
\end{theorem}

The following result will be used to apply Krasner's Lemma by taking approximations of separable polynomials.

\begin{theorem}[{\cite[Theorem 2.4.7]{pre05}}]
    \label{thm: approximate polynomials}
    Let $(K,v)$ be a valued field and lef
    $$f(X)=a_0+\cdots+a_{n-1} X^{n-1}+ X^n\in K[X]$$
    be a polynomial with distinct roots $x_1,\ldots,x_n\in K$. Let $\alpha\in vK$. Then, there exists $\gamma\in vK$ such that for 
    any polynomial
    $$g(X)=\prod_{i=1}^n(X-y_i)=b_0+\cdots+b_{n-1} X^{n-1}+ X^n$$
    with $y_1,\ldots,y_n\in K$ such that 
    $$\min_{0\le i< n}\{v(a_i-b_i)\}>\gamma,$$
    we have that for any $i\in\{1,\ldots,n\}$ there exists a $j\in\{1,\ldots,n\}$ with $v(x_i-y_j)>\alpha$. Furthermore, if $\alpha\ge\max_{i\neq j}\{v(x_i-x_j)\}$, then there exists only one $j$ such that $v(x_i-x_j)>\alpha$.
    
\end{theorem}

This paper mainly considers henselian fields, for which we use the following definition.
\begin{definition}
A valued field $(K,v)$ is called \textbf{henselian} if the valuation $v$ admits a unique extension to any algebraic extension of $K$.
\end{definition}

We will also use the following equivalent characterisation (see for example \cite[Theorem 4.1.3]{pre05}).

\begin{theorem}
    \label{thm: hensel}
    The following are equivalent.
    \begin{enumerate}
        \item $(K,v)$ is henselian.
        \item Let $f,g,h\in K_{\ge 0}[X]$ satisfy $\bar{f}=\bar{g}\bar{h}$, with $\bar{g}$, $\bar{h}$ relatively prime in $Kv[X]$. Then, there exist $g_1, h_1\in K_{\ge 0}[X]$ with $f=g_1h_1$, $\xoverline{g_1}=\bar{g}$, $\xoverline{h_1}=\bar{h}$ and $\deg{g_1}=\deg{\bar{g}}$.
        \item For each $f\in K_{\ge 0}[X]$ and $\alpha\in Kv$ with $\bar{f}(\alpha)$ and $\bar{f'}(\alpha)\neq 0$, there exists an element $a\in K_{\ge 0}$ such that $\bar{a}=\alpha$ and $f(a)=0$.
    \end{enumerate}
\end{theorem}

We will often use the following, appearing in Theorem 5.2.2. and Theorem 5.2.5. in \cite{pre05}.

\begin{theorem}
For a valued field $(K,v)$, there is a minimal immediate henselian algebraic extension $K^h$ of $K$, called the \textbf{henselisation} of $K$, which is unique up to valuation preserving isomorphism over $K$.\footnote{Minimal in the sense that $K^h$ embedds uniquely over $(K,v)$ into any henselian extension of $(K,v)$.}
\end{theorem}

With $K^s$ being the separable closure of $K$ and $v_s$ being an extension of $v$ to $K_s$, the henselisation $K^h$ of $K$ is defined as the fixed field of the subgroup of the absolute Galois group $G(K^s/K)$ which preserves the valuation ring $K^s_{\ge 0}$. In particular, $K^h$ is contained in $K^s$.

The \textbf{rank} of $v$ is the rank of $vK$, i.e. the number of proper convex subgroups of $vK$. When $v$ is of rank one, then $(\widehat{K}, v)$ is henselian (see for example \cite[Proposition 2.1.1]{pre05} together with \cite[Proposition 1.2.2]{pre05}).

\subsection{Algorithms, languages and model theory}
\label{sec: alg lang mod}
An \textbf{alphabet} is a non-empty set of symbols. A \textbf{string} over $\Sigma$ is a finite sequence with elements in $\Sigma$. We denote by $\Sigma^*$ the set of strings over $\Sigma$. If $s=s_1\cdots s_n\in \Sigma^*$, we will sometimes consider expressions of the form $s_1\dots s_0$ when iterating over substrings of $s$. This should be understood as the empty string. A subset of $\Sigma^*$ is called a \textbf{formal language} over $\Sigma$. We use the term formal language to distinguish from the notion of language in first order logic. With an \textbf{algorithm}, we mean a Turing machine, or any other equivalent model of computation. We will write algorithms in pseudocode using the syntax as exemplified here.
\begin{algorithm}[H]
\renewcommand{\thealgorithm}{}
\begin{algorithmic}
\caption*{\textsc{an\_algorithm}(input with specifications)}
\LineComment{A comment.}
\State $x\gets 0$ \Comment{Assign $0$ to the variable $x$.}
\While{$x\le n$} \Comment{Verify condition on $x$.}
    \For{$k\in \{x,\ldots,n\}$} \Comment{Iterate over elements.}
        \If{$P(x)$}
            \State{\Return $Q$}
        \EndIf
    \EndFor
    \State $x\gets x+1$
\EndWhile
\end{algorithmic}
\end{algorithm}

In many situations, we will define algorithms that implicitly depend on a parameter, for example a prime number $p$. When it is clear from context, we will not mention this parameter when using the algorithm, but rather assume that the correct version of the algorithm is used.

\begin{definition}
\label{def: recursively enumerable}
Let $\L$ be a formal language over a finite alphabet $\Sigma$. We say that $\L$ is \textbf{recursively enumerable} if there is an algorithm with input alphabet $\Sigma$ which on input $w\in\Sigma^*$ returns $\true$ if and only if $w\in L$.
\end{definition}

We will use this alternative characterisation of recursively enumerable, which is essentially Theorem 3.13 in \cite{sipser13}. For this, we identify $\N$ with the set of strings over the alphabet $\{*\}$ (or any alphabet of size $1$), by mapping $n\in\N$ to the string $*\cdots *$ of length $n$.

\begin{lemma}
\label{lem: recursively enumerable}
A formal language $\L$ over a finite alphabet $\Sigma$ is recursively enumerable if and only if there is an algorithm $\enumeratealg$ with input alphabet $\{*\}$ which on input $n\in \N$ returns an element $s\in \L$, and such that $\enumeratealg$ is surjective as a function from $\N$ to $\L$.
\end{lemma}

\begin{proof}
Suppose that $\L$ recursively enumerable and let $\ALG$ be the corresponding algorithm from Definition \ref{def: recursively enumerable}. Let $\Sigma=\{s_1,\ldots,s_m\}$ with $\lvert \Sigma\rvert = m$. Let $(p_n)_{n\ge 1}$ be the sequence of all prime numbers in increasing order. Consider the encoding of $\Sigma^*$ which associates a string $\prod_{j=1}^k s_{i_j}\in\Sigma^*$ with the unique natural number $\prod_{j=1}p_j^{p_{i_j}}$. This encoding induces a well-order $<_\Sigma$ on $\Sigma^*$. Write $\Sigma^*=\{w_i \ | \ i\in\N\}$, where $w_i<_\Sigma w_j$ whenever $i<j$. For $n\in\N$, we now define $\enumeratealg(n)$ as follows. Let $R$ be an empty list, to which we will add strings $w\in \L$.\footnote{More precisely, $R$ is a blank tape on a two tape Turing machine which only writes strings to $R$ in one direction, and uses some delimiter to distinguish between strings.} For $i=1,2,3,\ldots \ $, run $\ALG$ for $i$ steps on inputs $w_1,\ldots,w_i$. At each iteration, add to $R$ any such string for which $\ALG$ returns $\true$ after $i$ steps. Continue this until $R$ contains $n$ strings. By construction, this procedure will halt. The output of $\enumeratealg(n)$ is then the $n$:th string on $R$. If $w\in \L$, then $w=\enumeratealg(n)$ for some $n\in \N$. In fact, $w=\enumeratealg(n)$ for infinitely many $n$, since if $\ALG$ returns $\true$ on input $w$ after $i$ steps, it also returns $\true$ on input $w$ after $j$ steps for all $j>i$.

Conversely, suppose that $\enumeratealg$ is as described. We can then repeatedly compare $w$ to $\enumeratealg(n)$ for $n\in\N$. We define $\ALG$ to return $\true$ on input $w\in\Sigma^*$ if (and only if) $w=\enumeratealg(n)$ for some $n\in\N$. Since $\enumeratealg$ is surjective as a function from $\N$ to $\L$, we get that $\L$ is recursively enumerable.
\end{proof}

In the situation of Lemma \ref{lem: recursively enumerable}, we say that $\enumeratealg$ \textbf{enumerates} $\L$.

\begin{definition}
Let $\L$ be a formal language over a finite alphabet $\Sigma$. We say that $\L$ is \textbf{decidable} if there is an algorithm with input alphabet $\Sigma$ which on input $w\in\Sigma^*$ returns $\true$ if $w\in L$ and $\false$ if $w\notin L$.
\end{definition}

It follows that if $\L$ is decidable, then $\L$ is recursively enumerable by the same algorithm. We note that for any finite alphabet $\Sigma$, the set of strings $\Sigma^*$ is trivially decidable by an algorithm which returns $\true$ on every input.

For a finite field $\F_q$, where $q=p^n$ for some prime $p$ and some positive integer $n$, we will view $\F_q$ as a finite alphabet by fixing an irreducible polynomial $f$ of degree $n$ over $\F_p$. The alphabet $\F_q$ then consists of expressions of the form $a_0+a_1X+\cdots + a_{n-1}X^{n-1}$, where $a_i\in\{0,\ldots,p-1\}$.

For model theoretic terminology, we mostly follow the conventions in \cite{tenzie}. To clarify the two somewhat conflicting notions of language that we use, a first-order language $\L$ is an alphabet and the set of $\L$-sentences is a formal language over $\L$. We will refer to first-order languages simply as languages. The languages we will consider can be seen as finite alphabets by using $*$ instead of $\N$ when necessary. For example, $X_n$ will be an abbreviation for the string $X*\cdots *$ of length $n+1$.

If $\L$ is a language and $A$ and $B$ are $\L$-structures, we write $A\equiv B$ if $A$ and $B$ have the same $\L$-theories. Furthermore, if $C$ is a set which embeds in both $A$ and $B$, we write $A\equiv_C B$ if $A$ and $B$ have the same $\L(C)$-theories, where $\L(C)$ denotes the language $\L$ extended by adjoining constant symbols for the elements in $C$, interpreted by the respective embeddings into $A$ and $B$. In this situation, we will often for simplicity say that $C$ is a common subset of $A$ and $B$ and assume the embeddings to be inclusions. If $C$ is a singleton $\{c\}$, we write $A\equiv_c B$ instead of $A\equiv_C B$.

For an $\L$-structure $A$ and a subset $C$ of $L$, denote by $\acl_A(C)$ the model theoretic algebraic closure of $C$ in $A$. The following lemma is mentioned in the context of fields in \cite[p. 92]{dittmann2018a}. The provided proof follows closely the proof of Lemma 5.6.4 in \cite{tenzie} and was suggested by Ehud Hrushovski.

\begin{lemma}
\label{lem: acl}
Let $\L$ be a language and let $A$ and $B$ be $\L$-structures with a common subset $C$ such that $A\equiv_C B$. Then, there is a bijection between $\acl_A(C)$ and $\acl_B(C)$ such that $A\equiv_{\acl_A(C)}B$.  
\end{lemma}

\begin{proof}
Let $f$ be a partial embedding from $\acl_A(C)$ to $\acl_B(C)$ extending the identity on $C$. Consider $B$ as an $\L(C')$-structure by interpreting any constant symbol corresponding to $c\in C'$ by $f(c)$. Suppose that the domain $C'$ of $f$ is maximal with respect to inclusion and suppose that $f$ is such that $A\equiv_{C'}B$. Since the identity on $C$ satisfies these properties, we have that such a map $f$ exists.

Let $a$ be a tuple of elements in $\acl_A(C)$ and let $M$ be a saturated elementary extension of $A$. Since $a$ is algebraic over $C'$, there is an $\L(C')$-formula $\phi(X)$ such that $A\models\phi(a)$. Take $\phi$ to be minimal, in the sense that the cardinality of $\phi(M)$ is minimal. We then have that $\phi$ isolates the type $\tp(a/C')$. Indeed, suppose that there is an element $a'\in \phi(M)$ such that $\tp(a/C')\neq\tp(a'/C')$. Then, there is an $\L(C')$ formula $\phi'$ such that $a\in\phi'(M)$ but $a'\notin\phi'(M)$. Hence, $(\phi\land\phi')(M)$ is a non-empty proper subset of $\phi(M)$, contradicting minimality of $\phi$. Since $A\equiv_{C'} B$, there is an element $b\in \phi(B)$. Since $\phi$ isolates $\tp(a/C')$, we have for any $\psi\in \tp(a/C')$ that $A\models\forall X(\phi(X)\rightarrow \psi(X))$. This implies that $B\models\forall X(\phi(X)\rightarrow \psi(X))$ as well, so $\tp(b/C')=\tp(a/C')$. Thus, we can extend $f$ by sending $a$ to $b$, so $a\in C'$ by maximality of $f$. This shows that $C'=\acl_A(C)$ and $f$ is the embedding we are looking for.
\end{proof}

We denote by $\Lring=\{+,-,\cdot,0,1\}$ the language of fields and by $\Log=\{+,<,0\}$ the language of ordered groups. We extend $\Lring$ to the language of valued fields $\Lval=\{+,-,\cdot,^{-1},0,1,\vardiv\}$, where $\vardiv$ is a binary relation symbol. For a valued field $(K,v)$ and elements $a,b\in K$, the relation $\vardiv(a,b)$ is interpreted as $v(a)\ge v(b)$. While $v$ is not a symbol in $\Lval$, we will for simplicity write $v(X)=0$ as a shorthand for the $\Lval$-formula $\vardiv(1,X) \land \vardiv(X,1)$, since this formula defines the set of all elements with valuation $0$ in $(K,v)$. When we talk about the theories of $Kv$ and $vK$, we mean the $\Lring$-theory and $\Log$-theory respectively, if not stated otherwise.

We let $\L_t=\L_{\val}\cup\{t\}$, where $t$ is a constant symbol not in $\L_{\val}$. A Hahn field $\F((t^\Gamma))$ containing $t$ is an $\L_t$-structure by considering the $t$-adic valuation on $\F((t^\Gamma))$, as described in Section \ref{sec: hahn fields}. 

If $\L$ is a recursively enumerable language, then the set of $\L$-sentences is also recursively enumerable. If $A$ is an $\L$-structure, we say that $A$ is decidable if its $\L$-theory is decidable. We will use the fact that a recursively enumerable complete theory is decidable, which is seen by using the algorithm which enumerates $T$ to determine which one of $\phi$ and $\neg\phi$ are in $T$. If $T$ has a recursively enumerable axiomatisation, then $T$ is recursively enumerable, by listing the (finitely many) sentences with derivations of length at most $n$ that one can deduce from the first $n$ axioms of $T$, and then repeat for $n+1$, and so on.

Using results on the $\Log$-theories of ordered abelian groups in \cite{MR114855}, we get the following lemma.
\begin{lemma}
\label{lem: ord grps decidable}
The $\Log(1)$-theories of $\frac{1}{p^\infty}\Z$ and $\Q$ are decidable.
\end{lemma}

\begin{proof}
Let $\Log^+$ be the language obtained by adding to $\Log$ predicates $D_n$ for each $n$ coprime to $p$. These predicates are intepreted as
$$D_n(x) \Leftrightarrow \exists y(ny=x).$$
Both $\frac{1}{p^\infty}\Z$ and $\Q$ are regularly dense, in the sense of \cite{MR114855}. That is, they are archimedean and have no smallest positive element.
From Theorem 4.6 in \cite{MR114855} and its preceding discussion, we have recursively enumerable axiomatisations of the $\Log^+$-theories of $\frac{1}{p^\infty}\Z$ and $\Q$. From the proof of the same theorem, they are also model complete. Call these theories $T_{1/p^\infty}$ and $T_{\Q}$ respectively. In particular, all models of $T_{1/p^\infty}$ are $p$-divisible and all models of $T_{\Q}$ are divisible. We now claim that the $\Log^+(1)$-theories of $\frac{1}{p^\infty}\Z$ and $\Q$ are axiomatised by
$$T_{1/p^\infty} \cup \{1>0\}\cup \left\{\forall X(nX\neq 1) \ | \ n\in\N \setminus p\N \right\}$$ and
$$T_{\Q} \cup \{1>0\}$$ respectively. Indeed, any model of these extended theories contain $\frac{1}{p^\infty}\Z$ respectively $\Q$ as submodels. By model completeness of the $\Log^+$-theories, these are elementary substructures as $\Log^+$-structures. Hence, they are also elementary substructures as $\Log^+(1)$-structures. In other words, they are prime models of the respective $\Log^+(1)$-theories, and so the theories are complete. Since we added a recursively enumerable set of sentences to the recursively enumerable axiomatisations, we have that the $\Log^+(1)$-theories also have recursively enumerable axiomatisations. Hence they are recursively enumerable and decidable, as described above. Since the $\Log(1)$-theories are subsets of these theories, the result follows.
\end{proof}

\subsection{Hahn fields}
\label{sec: hahn fields}
The main objects of study in this thesis are Hahn fields, or fields of generalised power series. Introduced in \cite{hahn1907}, they are constructed in the following way. Let $\F$ be a field, let $\Gamma$ be an ordered abelian group and let $t$ be transcendental over $\F$. A generalised powers series in $t$ with coefficients in $\F$ and exponents in $\Gamma$ is a formal expression of the form
$$x=\sum_{\gamma\in\Gamma}a_\gamma t^\gamma$$
where the support, i.e. the set $\{\gamma\in\Gamma \ | \ a_\gamma\neq 0\}$, is well-ordered. The Hahn field $\F((t^\Gamma))$ is then the set of all such expressions together with term-wise addition and multiplication defined by
$$\left(\sum_{\gamma\in\Gamma}a_\gamma t^\gamma\right)\left(\sum_{\gamma\in\Gamma}b_\gamma t^\gamma\right)=\sum_{\gamma\in\Gamma}\sum_{\alpha+\beta=\gamma}a_{\alpha}b_{\beta} t^\gamma.$$
Note that multiplication is well defined since the supports are well-ordered. As the name suggests, $\F((t^\Gamma))$ is a field under these operations. 

For a generalised power series $x=\sum_{\gamma\in\Gamma}a_\gamma t^\gamma$, we will interchangeably use the notations
\begin{align*}
x=\sum_{\gamma\ge\gamma_0}a_\gamma t^{\gamma}, \hspace{1.5em} x=\sum_{i\in I}a_it^{\gamma_i}
\end{align*}
where $\gamma_0$ is minimal such that $a_{\gamma_0}\neq 0$ and $I$ is a well-ordered index set. Throughout, we will assume that $t\in\F((t^\Gamma))$. More precisely, this amounts to choosing a positive element $\gamma\in\Gamma$ and defining $t=t^{\gamma}$. Just as fields of formal Laurent series, $\F((t^\Gamma))$ admits a natural valuation $v$ by setting
$$v\left(\sum_{\gamma\ge\gamma_0}a_\gamma t^\gamma\right)=\gamma_0.$$
Under this valuation, $\F$ is the residue field and $\Gamma$ is the value group.

A standard result (see for example Theorem 1 in \cite{MR1225257}) is that $\F((t^\Gamma))$ is maximal. It follows that $\F((t^\Gamma))$ algebraically closed if $\F$ is algebraically closed and $\Gamma$ is divisible \cite[Corollary 4]{MR1225257}.  If $\F$ is algebraically closed and $\Gamma$ is divisible and non-trivial, then $\F((t^\Gamma))$ is universal, i.e.\ any field with the same cardinality and same characteristic embeds as a subfield in $\F((t^\Gamma))$ \cite{MR610}. More useful for us will be to consider algebraic extensions of $\F(t)$ inside the Hahn field $\F((t^\Gamma))$ as follows.

\begin{lemma}
\label{lem: rel alg closure in Hahn field}
Let $\F$ be a field of characteristic $p$ and let $\Gamma$ be a non-trivial ordered abelian group. Let $K$ be the relative algebraic closure of $\Fp(t)$ in $\F((t^\Gamma))$. Then $K$ is contained in $Kv((t^{vK}))$. 
\end{lemma}

\begin{proof}
Let $\Gamma'$ be the divisible hull of $\Gamma$. We can view $\Q$ as a subgroup of $\Gamma'$, by using the general assumption that that $t\in\F((t^\Gamma))$ and identifying $\Q$ with the divisible hull of $\langle v(t)\rangle$ in $\Gamma'$. Consider $\FpbarQ$ and $\F((t^\Gamma))$ as subfields of the Hahn field $\Fbar((t^{\Gamma'}))$. Since $\FpbarQ$ is algebraically closed and contains $\Fp(t)$, we have that $K$ is a subfield of
$$\FpbarQ\cap \F((t^\Gamma))=E((t^G)),$$
where $E=\Fpbar\cap\F$ and $G$ is the relative divisible hull of $v(t)$ in $\Gamma$. In particular, $Kv$ is contained in $E$, since $Kv$ is algebraic over $\Fp$. On the other hand, since $E$ is an algebraic extension of $\Fp$ inside $\F$, we also have that $E$ is contained in $Kv$. We thus get that $E=Kv$ and the result follows since $vK=G$.
\end{proof}

It was noted by Abhyankar \cite{MR80647} that generalised power series with coefficients in $\F_p$ and exponents in $\frac{1}{p^\infty}\Z$ arise naturally as root to the Artin-Schreier polynomial
$$f(X)=X^p-X-1/t\in\F_p(t).$$
Indeed, by linearity of the Frobenius, we have that
$$x=\sum_{n\ge 1}t^{-1/p^n}$$
is a root of $f$. We get all roots of $f$ by adding to $x$ elements of $\F_p$. It is tempting to see $x$ as some kind of limit to the sequence 
$$\left\{\sum_{i=1}^n t^{-1/p^i}\right\}$$
in $\F_p((t))^{1/p^\infty}$, the perfect hull of $\F_p((t))$. In this sense, $\PPFp$ could be seen as a completion to $\F_p((t))^{1/p^\infty}$. However, uniqueness fails since the sequence is not convergent and any element $x+y$ with $v(y)\ge 0$ could be seen as a limit. For this to make sense, we need a weaker notion of convergence, that of a pseudo-convergence, which will be made more precise in Section \ref{sec: kaplansky}. 

As we will see, the question of decidability for $\PPFp$ and similar fields reduces to the question of finding a decision procedure for the existence of roots to one variable polynomials over $\F_p(t)$. Such a procedure already exists for $\Fp((t))\perf$ by standard valuation theory. More precisely, any root to a polynomial $f\in\Fp(t)[X]$ in $\Fp((t))\perf$ is in $\Fp((t^{1/p^n}))$, where $p^n\ge\deg(f)$. Determining if $\Fp((t^{1/p^n}))$ has a root of $f$ can be done for example by the recursion procedure defined in \cite{lisinski2} and using Krasner's Lemma. Since $\PPFp$ is a maximal immediate extension of $\Fp((t))\perf$, finding a decision procedure for $\PPFp$ amounts to determining which immediate extensions of $\Fp((t))\perf$ lie inside $\PPFp$, given the reduction to one variable polynomials over $\Fp[t]$. The following result, appearing as Corollary 5.10 in \cite{Ku16}, shows that passing from $\Fp((t))\perf$ to its completion does not give more information about immediate extensions inside $\PPFp$.

\begin{theorem}
\label{thm: henselian existentially closed}
 Let $(K,v)$ be a henselian valued field $(K, v)$ and let $\widehat{K}$ be its completion, as described in Section \ref{sec: valued fields}. Then $(K,v)$ is existentially closed in its completion $(\widehat{K},v)$ if
and only if the extension $\widehat{K}/ K$ is separable.
\end{theorem}

\subsection{Kaplansky fields}
\label{sec: kaplansky}
While $\PPFp$ can be seen as a sort of completion of $\F_p((t))\perf$, in the sense that for example the sequence 
$$\left(\sum_{i\le n}t^{-1/p^i}\right)_{n\ge 1}$$
is completed by the element $\sum_{n\ge 1}t^{-1/p^n}\in\PPFp$, this completion is not unique. In the current section, we will describe this notion of completion by overviewing the work by Kaplansky in
\cite{kaplansky1942}. We will see that this will also give a deeper understanding of immediate extensions of valued fields. All results without proofs in this section are due to Kaplansky.
\begin{definition}
Let $(K, v)$ be a valued field and let $\{a_\rho\}_{\rho\in I}$ be a set of
elements of $K$ where the index set $I$ is well-ordered without a last element. Then $\{a_\rho\}$ is said to be \textbf{pseudo-convergent} if for all $\rho<\sigma<\tau$, we have
$v(a_\sigma - a_\rho) < v(a_\tau - a_\sigma)$.
\end{definition}

\begin{lemma}[{\cite[Lemma 1]{kaplansky1942}}]
\label{lem: pc stabilising or increasing}
Let $\{a_\rho\}$ be a pseudo-convergent sequence. Then one of the following holds.
 \begin{enumerate}
     \item $v(a_\rho) = v(a_\sigma)$ for all large enough $\sigma$ and $\rho$;
     \item $v(a_\rho) < v(a_\sigma)$ for all $\rho < \sigma$.
 \end{enumerate}
\end{lemma}

\begin{lemma}[{\cite[Lemma 2]{kaplansky1942}}]
\label{lem: difference stable}
Let $\{a_\rho\}$ be a pseudo-convergent sequence. Then
$$v(a_\sigma - a_\rho) = v(a_{\rho+1} - a_\rho)$$
for all $\sigma > \rho$.
\end{lemma}

By Lemma \ref{lem: difference stable}, we can write $\gamma_\rho := v(a_\sigma - a_\rho)$ for any $\sigma > \rho$. This gives us the following
definition.

\begin{definition}
We say that $\xi$ is a limit of a pseudo-convergent sequence $\{a_\rho\}$ if
$v(\xi - a_\rho) = \gamma_\rho$ for all $\rho$.
\end{definition}

The following important result establishes the first connection between pseudo-convergent
sequences and immediate extensions.

\begin{theorem}[{\cite[Theorem 1]{kaplansky1942}}]
\label{thm: immediate is pc limit}
Let $(L, v)$ be an extension of $(K, v)$. If $(L,v)/(K,v)$ is an immediate extension, then any element
in $L \setminus K$ is a limit of a pseudo-convergent sequence of elements in $K$, without a limit in $K$.
\end{theorem}

\begin{lemma}[{\cite[Lemma 5]{kaplansky1942}}]
\label{lem: pc type}
Let $\{a_\rho\}$ be a pseudo-convergent sequence in $(K, v)$ and let $f(X) \in
K[X]$ be a non-constant polynomial. Then for some sufficiently large $\lambda$, we have that $\{f(a_\rho)\}_{\rho>\lambda}$ is a pseudo-convergent sequence. In particular, one of the following holds.
\begin{enumerate}
\item $vf(a_\rho) = vf(a_\sigma)$ for all sufficiently large $\rho$ and $\sigma$;
\item $vf(a_\rho) < vf(a_\sigma)$ for all sufficiently large $\rho$ and $\sigma$, with $\rho < \sigma$.
\end{enumerate}
\end{lemma}

\begin{definition}
A pseudo-convergent sequence $\{a_\rho\}$ is said to be of \textbf{algebraic type}
if there is a polynomial $f(X) \in K[X]$ such that (2) in Lemma \ref{lem: pc type} holds for $f$. If $\{a_\rho\}$ is not of algebraic type, it is said to be of \textbf{transcendental type}.
\end{definition}

We can associate a kind of minimal polynomial $q(X) \in K[X]$ to a pseudo-convergent
sequence of algebraic type by letting $q$ be of minimal degree such that (2) in Lemma \ref{lem: pc type} holds for $q$.

\begin{theorem}[{\cite[Theorem 2 and Theorem 3]{kaplansky1942}}]
\label{thm: pc type}
Let $\{a_\rho\}$ be a pseudo-convergent sequence without a limit in $K$. Then the following hold.
\begin{enumerate}
    \item If $\{a_\rho\}$ is of transcendental type, then there is an immediate transcendental extension $K(z)$ of $K$ such that $z$ is a limit of $\{a_\rho\}$. Furthermore, if $K(u)$ is an immediate extension of $K$ such that $u$ is a limit of $\{a_\rho\}$, then $K(u)$ is isomorphic to $K(z)$ over $K$.
    \item If $\{a_\rho\}$ is of algebraic type and $q(x)$ is a minimal polynomial of $\{a_\rho\}$, in the sense described above, then there is an immediate algebraic extension $K(z)$ of $K$ such that $z$ is a limit of $\{a_\rho\}$ and $q(z) = 0$. Furthermore, if $K(u)$ is an immediate extension of $K$ such that $u$ is a limit of $\{a_\rho\}$ and $q(u) = 0$, then $K(u)$ is isomorphic to $K(z)$ over $K$.
\end{enumerate}
\end{theorem}

\begin{remark}
Note that Theorem \ref{thm: pc type} does not say that the limit of a pseudo-convergent sequence of algebraic type necessarily defines an algebraic extension. Indeed, consider the sequence $\sum_{i=1}^n t^{-1/p^i}$ over $\F_p((t))\perf$. This is a pseudo-convergent sequence with the limit $x=\sum_{n\ge 1} t^{-1/p^n}\in\PPFp$. Then, for any transcendental element $y\in\PPFp_{\ge 0}$, we have that $x+y$ is a transcendental pseudo-limit of the given sequence.
\end{remark}

For a valued field $(K, v)$ of equal characteristic $p$, the following two conditions are what Kaplansky calls \textbf{Hypothesis A} \cite{kaplansky1942}, which provides a criterion for understanding all the immediate extensions of $K$.
\begin{enumerate}
    \item Any non-zero additive polynomial $f \in Kv[X]$ is surjective on $Kv$.
    \item The value group $vK$ is $p$-divisible.
\end{enumerate}
A valued field satisfying Hypothesis A is also called a \textbf{Kaplansky field}. An immediate but important fact is that any immediate extension of a Kaplansky field is also a Kaplansky field. Furthermore, note that the conditions of Hypothesis A can be expressed as first order statements in $\Lring$ and $\Log$ respectively. Therefore, if $(L,v)$ is a Kaplansky field and $(F,w)$ is a valued field such that $Lv\equiv Fw$ and $vL\equiv wF$, we have that $(F,w)$ is also a Kaplansky field. The importance of Hypothesis A is captured in the following Theorem.

\begin{theorem}[{\cite[Theorem 5]{kaplansky1942}}]
\label{thm: unique maximal}
Let $(K, v)$ be a valued field of equal characteristic $p$ satisfying Hypothesis A. Then, $(K,v)$ admits a maximal immediate extension $(L,v)$ which is unique up to valuation preserving isomorphism over $K$.
\end{theorem}

We will mainly be interested in uniqueness of immediate algebraic algebraically maximal extensions. For a Kaplansky field, this situation is implicitly covered in the proof of Theorem \ref{thm: unique maximal}. The following result shows however that we do not need to assume hypothesis A, as long as the maximal immediate extension is unique.

\begin{lemma}
\label{lem: unique algebraically maximal}
Let $(E,v)$ be a valued field admitting a maximal immediate extension $(L,v)$ which is unique up to valuation preserving isomorphism over $E$. Then, the relative algebraic closure $(K,v)$ of $(E,v)$ in $(L,v)$ is the unique, up to valuation preserving isomorphism over $E$, algebraic extension of $(E,v)$ which is immediate and algebraically maximal.
\end{lemma}

\begin{proof}
Let $K'$ be an immediate algebraic algebraically maximal extensions of $E$. It is enough to show that $K$ is isomorphic to $K'$ as a valued field over $E$, since this implies that $K$ is also algebraically maximal. Let $L'$ be a maximal immediate extension of $K'$. Then $K'$ is the relative algebraic closure of $E$ in $L'$, since any extension of $K'$ inside $L'$ is immediate and since $K'$ is algebraically maximal. Furthermore, we have that $L$ and $L'$ are isomorphic as valued fields over $E$, by assumption on $E$. Denote this isomorphism by $\Phi$. For any polynomial $f\in E[X]$, we have that $a\in L$ is a root of $f$ if and only if $\Phi(a)\in L'$ is a root of $f$. Therefore $\Phi(K)=K'$ and we are done.
\end{proof}

The following result by Whaples gives an alternative useful characterisation of Kaplansky fields. Note that the original statement in \cite[Theorem 1]{whaples1957} refers to the residue field as a Kaplansky field, rather than the valued field itself.

\begin{theorem}
\label{thm: whaples}
A valued field $(K,v)$ is a Kaplansky field if and only if it satisfies the following.
\begin{enumerate}
    \item[1'.] $Kv$ has no algebraic extension of degree divisible by $p$.
    \item[2.] The value group $vK$ is $p$-divisible.
\end{enumerate}
\end{theorem}

We will need the following results, which appear in Theorem 1.1 and Proposition 1.2 in \cite{kuhlmann2018subfields}.
\begin{theorem}
\label{thm: relative algebraic closure in maximal kaplansky}
 Let $(L,v)$ be an algebraically maximal Kaplansky field. Then for any subfield $K$ of $L$, we have that $L$ contains an algebraically maximal immediate extension of $K$. Furthermore, if $K$ is a relatively algebraically closed subfield of $L$, then $(K, v)$ is also an algebraically maximal Kaplansky field.
\end{theorem}

\subsection{Tame fields}

When $\F$ is a perfect field and $\Gamma$ is a $p$-divisible ordered abelian group, then $\F((t^\Gamma))$ falls in a class of fields called tame fields. This class was studied extensively by Kuhlman in \cite{Ku16} and we follow this approach.

\begin{definition}
An algebraic extension $(L|K,v)$ of a henselian valued field $(K,v)$ is called \textbf{tame} if every finite subextension $E| K$ of $L| K$ satisfies the following conditions:
\begin{enumerate}
\item $(vE:vK)$ is prime to $p$
\item $Ev| Kv$ is separable
\item $E| K$ is defectless.
\end{enumerate}
A \textbf{tame field} is a henselian valued field for which all algebraic extensions are tame.
\end{definition}

That Hahn fields with perfect residue field and $p$-divisible value group are canonical examples of tame fields is, if not apparent from the definition, clear from the following alternative characterisations \cite[Theorem 3.2, Corollary 3.3]{Ku16}.

\begin{theorem}
\label{thm: hahn condition for tame}
A valued field $(K,v)$ is tame if and only if $(K,v)$ is algebraically maximal, $Kv$ is perfect, and $vK$ is $p$-divisible. If $(K,v)$ has characteristic $(p,p)$, then $(K,v)$ is tame if and only if it is algebraically maximal and perfect.
\end{theorem}

Given Theorem \ref{thm: hahn condition for tame} and Kaplansky's characterisation of immediate extensions, the following result is clear for tame Hahn fields of equal characteristic. For the general situation, see \cite[Lemma 3.7]{Ku16}

\begin{lemma}
\label{lem: general tame rel condition}
Let $(L,v)$ be a tame field and let $(K,v)\subset (L,v)$ be a relatively algebraically closed subfield. Suppose that $Lv|Kv$ is an algebraic extension. Then $(K,v)$ is a tame field, $vL/vK$ is torsion free, and $Kv=Lv$.
\end{lemma}

For our purposes we note that the hypothesis in Lemma \ref{lem: general tame rel condition} can be weakened when working with Hahn fields, to obtain a similar result.

\begin{lemma}
\label{lem: hahn tame rel condition}
Let $\F$ be a perfect field of characteristic $p$ and let $\Gamma$ be a $p$-divisible value group. Let $(K,v)$ be the relative algebraic closure of $\F_p(t)$ in  $\F((t^\Gamma))$. Then $(K,v)$ is a tame field and $\Gamma/vK$ is torsion free.
\end{lemma}

\begin{proof}
Since $vK$ is the relative divisible hull of $v(t)$ in $\Gamma$, we have that $\Gamma/vK$ is torsion free. By Lemma \ref{lem: rel alg closure in Hahn field}, we have that $K$ is relatively algebraically closed in $L=Kv((t^{vK}))$. By assumption on $F$ and $\Gamma$, we have that $Kv$ is perfect and $vK$ is $p$-divisible, so $L$ is tame. Since $Lv=Kv$, we are in the situation of Lemma \ref{lem: general tame rel condition} and $(K,v)$ is a tame field.
\end{proof}

One of the main results in \cite{Ku16} is that tame fields admit an Ax-Kochen Ershov principle in the language $\L_{\val}$. In particular, this implies the following.

\begin{theorem}[{\cite[Theorem 1.6]{Ku16}}]
\label{thm: decidability Lval}
Let $q$ be a power of a prime $p$ and let $\Gamma$ be an ordered abelian group which is divisible or elementarily equivalent to $\frac{1}{p^\infty}\Z$. Then $\F_q((t^\Gamma))$ equipped with the $t$-adic valuation is decidable in the language $\Lval$.
\end{theorem}
To mimic this result for $\L_t$, we need the following.

\begin{definition}
\label{def: rel subcompl}
Let $\mathcal{C}$ be an elementary class of tame fields in the language $\Lval$. If for every two fields $(L, v),(F, w) \in \mathcal{C}$ and every common defectless subfield $(K, v)$ of $(L, v)$ and $(F, w)$ such that $vL/vK$ is torsion free and $Lv| Kv$ is separable, the conditions $vL \equiv_{vK} wF$ and $Lv  \equiv_{Kv} F w$ imply that $(L, v)  \equiv_{(K,v)}(F, w)$, then we will call $\mathcal{C}$ \textbf{relatively subcomplete} (in the language $\Lval$).
\end{definition}

\begin{theorem}[Theorem 7.1 in \cite{Ku16}]
\label{thm: rel subcompl}
The class of tame fields is relatively subcomplete in the language $\L_{\val}$.
\end{theorem}

\subsection{Finite automata}
\label{sec: finite automata}
In this section, we include some standard results on finite automata, as presented in \cite{MR1997038}. We also present key results in \cite{MR2289431}. The main findings of this paper rely on the existence of algorithms which outputs certain automata. The results are thus stated in this manner, i.e. by saying that there exists an algorithm that outputs an automaton with certain properties. While the referenced results are not stated explicitly like this, it follows from the proofs in the references that such algorithms do exist. When necessary, additional details are provided in Appendix \ref{appendix} to emphasise this.

We will need the following three kinds of finite automata.

\begin{definition}
\label{def: DFA}
A \textbf{deterministic finite automaton}, or a \textbf{DFA}, is a tuple $M=(Q,\Sigma,\delta,q_0,F)$ where
\begin{itemize}
    \item $Q$ is a finite set (the \textbf{states});
    \item $\Sigma$ is a finite set (the \textbf{input alphabet});
    \item $\delta$ is a function from $Q\times\Sigma$ to $Q$ (the \textbf{transition function});
    \item $q_0\in Q$ (the \textbf{initial state});
    \item $F$ is a subset of $Q$ (the \textbf{accepting states}).
\end{itemize}
\end{definition}

\begin{definition}
\label{def: DFAO}
A \textbf{deterministic finite automaton with output}, or a \textbf{DFAO}, is a tuple $M=(Q,\Sigma,\delta,q_0,\Delta,\tau)$ where
\begin{itemize}
    \item $Q$ is a finite set (the \textbf{states});
    \item $\Sigma$ is a finite set (the \textbf{input alphabet});
    \item $\delta$ is a function from $Q\times\Sigma$ to $Q$ (the \textbf{transition function});
    \item $q_0\in Q$ (the \textbf{initial state});
    \item $\Delta$ is a finite set (the \textbf{output alphabet});
    \item $\tau$ is a function from $Q$ to $\Delta$ (the \textbf{output function}).
\end{itemize}
\end{definition}

\begin{definition}
\label{def: NFA}
A \textbf{nondeterministic finite automaton}, or an \textbf{NFA}, is a tuple $M=(Q,\Sigma,\delta,q_0,F)$ where
\begin{itemize}
    \item $Q$ is a finite set (the \textbf{states});
    \item $\Sigma$ is a finite set (the \textbf{input alphabet});
    \item $\delta$ is a function from $Q\times\Sigma$ to the power set of $Q$ (the \textbf{transition function});
    \item $q_0\in Q$ (the \textbf{initial state});
    \item $F$ is a subset of $Q$ (the \textbf{accepting states}).
\end{itemize}
\end{definition}

If $M=(Q,\Sigma,\delta,q_0,\Delta,\tau)$ is a DFAO and $0\in \Delta$, we say that a state $q\in Q$ an \textbf{accepting state of $M$} if $\tau(q)\neq 0$. Thus, a DFA can be seen as a DFAO with output alphabet $\{0,1\}$.

We will consider families of DFAOs where the input alphabet $\Sigma$ and the output alphabet $\Delta$ are fixed. Denote the collection of such DFAOs by $\calD(\Sigma,\Delta)$. For any DFAO, we will identify its set of states $Q$ with the set $\{1,\ldots,\lvert Q\rvert\}$ if not specified otherwise. With this identification, we can view $\calD(\Sigma,\Delta)$ as a set. Furthermore, with the identifications
\begin{align*}
    s_1 & = \ ( \\
    s_2 & = \ \{ \\
    s_3 & = \ ,  \\
    s_4 & = \ \} \\
    s_5 & = \ )
\end{align*}
$\calD(\Sigma,\Delta)$ can be seen as a formal language $\L$ over the alphabet
$$\Xi=\left\{s_1,s_2,s_3,s_4,s_5\right\}\cup\N\cup \Sigma\cup \Delta.$$
This alphabet is recursively enumerable by an algorithm $e$, since the only infinite set in the union is $\N$. We can also choose $e$ such that it is a bijection as a function from $\Xi$ to $\N$, and such that its inverse is given by an algorithm $e^{-1}$. By this, $\Xi^*$ is recursively enumerable, using the encoding described in Section \ref{sec: alg lang mod}. 

For simplicity, we enforce any string in $\L$ to be without repetitions of the elements in the sets, and with elements of $Q$ in increasing order, i.e. on the form
$$\left(\{1,\ldots,i\},\{\Sigma_1,\ldots,\Sigma_j\},\{\delta_1,\ldots,\delta_k\},q_0,\{\Delta_1,\ldots,\Delta_\ell\},\{\tau_1,\ldots,\tau_m\}\right)$$
where $\lvert\{\{\Sigma_1,\ldots,\Sigma_j\}\rvert=j$, $\lvert \{\delta_1,\ldots,\delta_k\} \rvert = k$, $\{\Delta_1,\ldots,\Delta_\ell\}\rvert = \ell$, and $\lvert\{\tau_1,\ldots,\tau_m\}\rvert=m$. By the choice of Gödel numbering, the set $G=\{\encode(w) \ | \ w\in \L\}$ is a decidable subset of $\N$. We thus get that $\L$ is enumerated by an algorithm $\listdfao$ which on input $n$ returns the string which is encoded as the $n$:th natural number in $G$. Since any DFAO in $\calD(\Sigma,\Delta)$ is represented by a string in $\L$ we have that $\listdfao$ enumerates $\calD(\Sigma,\Delta)$.

For a finite field $\F_q$, where $q=p^n$ for some prime $p$ and some positive integer $n$, we will view $\F_q$ as an alphabet by fixing an irreducible polynomial $f$ of degree $n$ over $\F_p$. The alphabet $\F_q$ then consists of expressions of the form $a_0+a_1X+\cdots + a_{n-1}X^{n-1}$, where $k<n$ and $a_i\in\{0,\ldots,p-1\}$.

If $M$ is a DFAO with input alphabet $\Sigma$ and transition function $\delta$, we extend $\delta$ to the function
$$\delta^*:Q\times\Sigma^*\to Q$$
defined recursively by
\begin{align*}
    \delta^*(q,\emptyset) & =q \\
    \delta^*(q,wa) & =\delta(\delta^*(q,w),a)
\end{align*}
where $q\in Q$, $w\in\Sigma^*$ and $a\in \Sigma$. Furthermore, we let
$$f_M:\Sigma^*\to\Delta$$ be the function defined by
$$f_M(w)=\tau(\delta^*(q_0,w)).$$

Note that any DFA is both a DFAO and an NFA, and that the following two definitions indeed agree on DFAs.

\begin{definition}
\label{def: accepting}
Let $M=(Q,\Sigma,\delta,q_0,\Delta,\tau)$ be a DFAO. For a string $w\in\Sigma^*$, we say that \textbf{$M$ accepts $w$} if $f_M(w)\neq 0$. The set of strings in $\Sigma^*$ accepted by $M$ is called the \textbf{language accepted by $M$}. A language accepted by some DFA is called \textbf{regular}.
\end{definition}

\begin{definition}
\label{def: accepting path}
Let $N=(Q,\Sigma,\delta,q_0,F)$ be an NFA and let $w=s_1\cdots s_n\in\Sigma^*$. An \textbf{accepting path} for $w$ is a sequence of states $q_1,\ldots, q_n \in Q$ such that $q_i \in \delta(q_{i-1}, s_i)$ for $i \in\{1,\ldots, n\}$ and $q_n\in F$. We say that \textbf{$N$ accepts $w$} if there exists an accepting path for $w$ in $N$. The set of strings in $\Sigma^*$ accepted by $N$ is called the \textbf{language accepted by $N$}. 
\end{definition}

As models of computation, an NFA is equivalent to a DFA in the following sense (see for example Theorem 4.1.3 in \cite{MR1997038}).

\begin{theorem}
\label{thm: NFA DFA equiv}
There is an algorithm $\dfa$ which takes as input an NFA $M$ and outputs a DFA $M'$ such that $M$ and $M'$ accept the same language. 
\end{theorem}

\begin{definition}
\label{def: minimal DFA}
Let $M=(Q,\Sigma,\delta,q_0,F)$ be a DFA. We say that $M$ is \textbf{minimal} if there is no DFA with fewer states than $M$ accepting the same language as $M$.
\end{definition}

\begin{remark}
\label{rem: minimal unique}
There is an algorithm which takes as input a DFA $M$ and outputs a minimal DFA $M'$ which accepts the same language as $M$ and which is unique up to renaming the states. For details, see Section 4.4.3 in \cite{Hopcroft+Ullman/79/Introduction}. We will denote this algorithm by $\mindfa$.
\end{remark}

The Myhill-Nerode theorem, appearing for example as Theorem 4.1.8 in \cite{MR1997038}, provides a way verify if a given language is regular not.

\begin{definition}
\label{def: nerode equivalence}
Let $\L$ be a language over an alphabet $\Sigma$. We say that an equivalence relation $\sim$ on $\Sigma$ is \textbf{right-invariant} if $w\sim w'$ implies $wz\sim w'z$ for all $z\in\Sigma^*$. 
\end{definition}

\begin{theorem}[Myhill-Nerode theorem]
\label{thm: myhill-nerode}
Let $\L$ be a language over a finite alphabet $\Sigma$. Then $\L$ is regular if and only if there exists a right-invariant equivalence relation $\sim$ of finite index on $\Sigma^*$ such that $\L$ is the union of some of the equivalence classes of $\sim$.
\end{theorem}

The notion of minimality for DFAs does not extend directly to DFAOs, since all the properties of a DFAO are not captured by the language it accepts. We will instead use this less restrictive notion of minimality.

\begin{definition}
\label{def: minimal DFAO}
Let $M=(Q,\Sigma,\delta,q_0,\Delta,\tau)$ be a DFAO. We say that a state $q\in Q$ is \textbf{reachable} from a state $q_1$ if there is some string $s\in\Sigma^*$ such that $\delta^*(q_1,s)=q$. If $q$ is reachable from $q_0$, we say that $q$ is reachable. If $q$ is not reachable, we say that it is \textbf{unreachable}. If $M$ has no unreachable states, we say that $M$ is \textbf{minimal}.
\end{definition}

Since removing unreachable states does not change the accepted language, if $M$ is a minimal DFA, then it is also minimal as a DFAO.

\begin{remark}
\label{rem: bounded strings to reach}
If $M=(Q,\Sigma,\delta,q_0,\tau,\Delta)$ is a DFAO with set of states $Q$ and if $q'\in Q$ is reachable from $q\in Q$ by a string $w=s_1\ldots s_n$, then we can assume that $n\le \lvert Q\rvert$. Indeed, if $\delta^*(q,s_1\ldots s_i)=\delta^*(q,s_1\ldots s_j)$ for some $i\neq j$, then we can replace $w$ with $w'=s_1\ldots s_i s_{j+1}\ldots s_n$, which still gives a path from $q$ to $q'$. Consequently, there is an algorithm $\reachablestates$ which on input $M$ returns all states of $M$ reachable states of $M$, by returning all states of $M$ reachable by strings of length at most $\lvert Q\rvert$.
\end{remark}

We will mainly consider finite automata with input alphabet $$\Sigma_p=\{0,1,\ldots,p-1,.\}.$$

\begin{definition}
\label{def: valid expansion}
A string $s=s_1\cdots s_n\in\Sigma_p^*$ is said to be a \textbf{valid base $p$-expansion} if $s_1\neq 0$, $s_n\neq 0$, and $s_k$ is equal to the radix point for exactly one $k\in \{1,\ldots,n\}$. If $s$ is a valid base $p$-expansion and $s_k$ is the radix point, then we define the \textbf{value} of $s$ to be
$$v(s)=\sum_{i=1}^{k-1}s_ip^{k-1-i}+\sum_{i=k+1}^ns_ip^{k-i}\in \frac{1}{p^\infty}\N.$$
Conversely, for an element $v\in\frac{1}{p^\infty}\N$, we denote by $s(v)$ the valid base $p$-expansion of $v$.
\end{definition}

\begin{definition}
\label{def: p-automatic}
Let $\Delta$ be a finite set. A function
$$f:\frac{1}{p^\infty}\N\to\Delta$$
is called \textbf{$p$-automatic} if there is a DFAO $M=(Q,\Sigma,q_0,\Delta,\tau)$ such that for any $v\in\frac{1}{p^\infty}\Z$, we have that $f(v)=f_M(s(v))$.
\end{definition}

The connection between finite automata and generalised power series is captured in the following result, which is a particular instance of the more general Theorem 4.1.3. in \cite{MR2289431}.

\begin{theorem}
\label{thm: algebraic iff p-automatic}
Let $q$ be a power of a prime $p$ and let $f:\frac{1}{p^\infty}\N\to\F_q$ be a function with well-ordered support. Then the generalised power series
$$\sum_{\gamma\in\frac{1}{p^\infty}\N}f(\gamma)t^\gamma\in\PRFq$$
is algebraic over $\F_q(t)$ if and only if $f$ is $p$-automatic.
\end{theorem}

\begin{definition}
\label{def: well-formed}
Let $M=(Q,\Sigma_p,\delta,q_0,\Delta,\tau)$ be a DFAO. We say that $M$ is \textbf{well-formed} (resp. \textbf{well-ordered}) if there is an arbitrary (resp. a well-ordered) subset $S\subset\frac{1}{p^\infty}\N$ such that the language accepted by $M$ consists of the valid base $p$ expansions of the elements of $S$.
\end{definition}

For a well-ordered DFAO with input alphabet $\Sigma_p$ and output set $\F_q$, we denote by $\Pow(M)$ the element
$$\sum_{\gamma\in\frac{1}{p^\infty}\N}f_M(s(\gamma))t^\gamma\in\PRFq.$$

The following results are implicit in \cite{MR2289431}. Detailed proofs are provided in Appendix \ref{appendix}.

\begin{lemma}
\label{lem: equals}
There is an algorithm $\equals$ which takes as input two well-ordered DFAOs $M$ and $N$ with output alphabet $\F_q$, returning $\true$ if $\Pow(M)=\Pow(N)$ and $\false$ otherwise
\end{lemma}

\begin{proof}
See Remark \ref{rem: equals}.
\end{proof}

\begin{lemma}
\label{lem: list well ord}
Let $q=p^n$ be a prime power. Then, there is an algorithm $\listwellord\_\F_q$ which enumerates well-ordered DFAOs with output alphabet $\F_q$.
\end{lemma}

\begin{proof}
    See Remark \ref{rem: list well-ordered}.
\end{proof}

\begin{lemma}
\label{lem: is root}
There is an algorithm $\isrootalg$ which takes as input a polynomial $f(X)\in\F_p[t][X]$ and a DFAO $\X$ and returns $\true$ if $f(\Pow(\X))=0$ and $\false$ otherwise.
\end{lemma}

\begin{proof}
See Remark \ref{rem: is root}.
\end{proof}

\section{The \texorpdfstring{$\L_t$}{[Lt]}-theories of tame Hahn fields}

\subsection{An AKE-principle for tame fields in \texorpdfstring{$\L_t$}{[Lt]}}
\label{sec: ake}
Throughout this section, let $(E,v)$ be a fixed valued field of residue characteristic $p$ such that $Ev$ is an algebraic extension of its prime field. Let $\Pi\subset E$ be such that $v\Pi$ generates $vE$. Let $(L,v)$ be a tame field containing $E$ and denote by $(K,v)$ the relative algebraic closure of $E$ in $L$. We will assume that $(K,v)$ is algebraically maximal and that $vL/vK$ is torsion free. In particular, if $E=\F_p(t)$ with the $t$-adic valuation, then $L$ can be an equal charactersitic tame Hahn field, since the conditions on $K$ are then satisfied by Lemma \ref{lem: hahn tame rel condition}.

We will use Theorem \ref{thm: rel subcompl} to obtain an AKE-principle for the theory of $(L,v)$ in the language $\Lval(v\Pi)$, following closely the proof of Lemma 6.1 in \cite{Ku16}.\footnote{This AKE-principle was originally formulated for $E=\F_p(t)$ and $\Pi=\{t\}$. The general result was prompted by Konstantinos Kartas asking if it also holds in characteristic zero. I would also like to thank him for suggesting to consider $E$ as an arbitrary subfield, as an earlier version (see \cite{lisinski}) only treated $E=\F_p(t)$ and $E=\Q$.}

\begin{lemma}
\label{lem: ake}
Let $(F,w)$ be a tame field containing $(E,v)$ and suppose that \linebreak $vL\equiv_{v\Pi} wF$ and $Lv\equiv_{Ev} Fw$. Suppose furthermore that $(K,v)$ is isomorphic over $(E,v)$ to a valued subfield of $(F,w)$. Then $(L,v)\equiv(F,w)$ in the language $\Lval(K)$.
\end{lemma}

\begin{proof}
By assumption, $(K,v)$ is algebraically maximal. Since $L$ is perfect, we also have that $K$ is perfect. Furthermore, since $vL$ is $p$-divisible and $vL/vK$ is torsion free, we have that $vK$ is $p$-divisible. Hence, by Theorem \ref{thm: hahn condition for tame}, we have that $(K,v)$ is tame. In particular, it is defectless. Since $Lv|Kv$ is separable we are in the situation of Definition \ref{def: rel subcompl} and it is enough to show that $vL\equiv_{vK} wF$ and $Lv\equiv_{Kv} Fw$.

Let $\psi(\bar{a})$ be a sentence in $\Log(vK)$. Since $K$ is an algebraic extension of $E$, we have that $vK$ is a subgroup of the divisible hull of $vE$. With $\bar{a}=(a_1,\ldots,a_m)$, we can therefore write
$$a_i=\sum_{j=1}^n\frac{b_{i,j}}{c_{i,j}}v(\pi_{i,j}),$$
where $\frac{b_{i,j}}{c_{i,j}}\in\Q$ and $\pi_{i,j}\in\Pi$. Each $a_i$ is the unique element in $vK$ satisfying the $\Log(v\Pi)$-formula $\phi_i(X)$ defined as
$$\prod_{j=1}^n c_{i,j}X =\sum_{j=1}^nb_{i,j}\prod_{k\neq j}c_{i,k} v(\pi_{i,j}).$$
Hence, $\psi(\bar(a))$ is equivalent to the $\Log(v\Pi)$-sentence
$$\exists X_1\cdots\exists X_m\left(\psi(X_1,\ldots,X_m)\bigwedge_{i=1}^m \phi_i(X_i) \right).$$
Since $vL$ and $wF$ are elementary equivalent in the language $\Log(v\Pi)$, they are therefore elementary equivalent in the language $\Log(vK)$.

Since $Kv$ is contained in $\acl(E)$, we have by Lemma \ref{lem: acl} that $Lv\equiv_{Kv}Fw$. We conclude that we are in the situation of Definition \ref{def: rel subcompl}. By Theorem \ref{thm: rel subcompl}, we get that $(L,v)\equiv_{(K,v)}(F,w)$ in the language $\Lval$, and we are done.
\end{proof}

We will now show that Lemma \ref{lem: ake} implies an AKE-principle in $\Lval(K)$ relative to the algebraic part. For a set $A\subset L$, denote by $\M_A$ the set of monic polynomials over $A$. For each $f\in\M_A$, let $\phi_f$ be the $\Lval(A)$-sentence defined as
$$\exists X \left(f(X)=0 \land v(X)\ge 0\right).$$
Define
$$S_{L,A}=\{\phi_f\mid f \in \M_A, \ (L,v)\models \phi_f\}.$$
We then have the following.

\begin{theorem}
\label{thm: ake without kaplansky}
Suppose that $(E,v)$ has rank one. Let $A\subset E_{\ge 0}$ be dense in the completion of $E_{\ge 0}$, with respect to $v$. Let $(F, w)$ be a tame field containing $(E,v)$ such that $(F,w) \models S_{L,A}$ and such that $vL\equiv_{v\Pi} wF$ and $Lv\equiv Fw$. Then, $(K,v)$ is isomorphic over $E$ to a subfield of $F$ and, identifying $K$ with its image under this isomorphism, we have that $(L,v)\equiv(F,w)$ in the language $\Lval(K)$.
\end{theorem}

\begin{proof}
Let $(F, w)$ be as described. We want to show that $(K,v)$ is isomorphic as a valued field over $E$ to a valued subfield of $F$. Since $(L,v)$ and $(F,w)$ are henselian, the isomorphism of $(E,v)$ and $(E,w)$ can be extended to an isomorphism of valued fields $\Phi$ over $E$ of the henselisations $E^h$ of $E$ in $L$ and $F$ respectively. Consider the set $\U$ of finite separable extensions of $E^h$ inside $L$. We claim that any element in $\U$ is imomorphic over $E^h$ to a subfield of $F$. Indeed, let $U\in\U$. Then $U=E^h(c_0)$, where $c_0$ is integral over $E^h_{\ge 0}$. Let 
$$u=a_0+\cdots + a_{n-1}X^{n-1} +X^n\in E^h_{\ge 0}[X]$$
be the minimal polynomial of $c_0$ and let $\{c_0,\ldots,c_{n-1}\}\subset\bar{E}$ be the conjugates of $c_0$. Let $\alpha_i=\max_{i\neq j}\{v(c_i-c_j)\}$ and let $\alpha=\max_i\{\alpha_i\}$. Since any valuation ring is integrally closed in its field of fraction we have that $v(c_i)\ge 0$ for each $i$, so $\alpha\ge 0$. By Theorem \ref{thm: approximate polynomials}, there is $\gamma\in v\bar{E}$ such that for any polynomial
$$\tilde{u}=\prod_{i=1}^n(X-\tilde{c}_i)=\tilde{a}_0+\cdots + \tilde{a}_{n-1}X^{n-1} +X^n\in \bar{E}[X]$$
with $\min_{0\le i<n}\{v(a_i-\tilde{a}_i)\}>\gamma$ and for any $i\in\{1,\ldots,n\}$, there is exactly one $j_i\in\{1,\ldots,n\}$ such that $v(c_i-\tilde{c}_{j_i})>\alpha$. In particular, any such $\tilde{u}$ is separable and $v(\tilde{c}_{j_0})\ge 0$. Since $(E,v)$ has rank one, we have that $E^h$ is contained in the completion of $E$ with respect to $v$. Hence, by assumption on $A$, we can let $\tilde{u}=\tilde{a}_0+\cdots + \tilde{a}_{n-1}X^{n-1} +X^n\in A[X]$ be such that $\min_{0\le i<n}\{v(a_i-\tilde{a}_i)\}>\gamma$. By Krasner's Lemma, we get that $E^h(c_0)\subset E^h(\tilde{c_{j_0}})$. By the degree of $c_0$, this implies that $\tilde{u}$ is irreducible over $E^h$. Conversely, we can take $\tilde{u}$ to be close enough to $u$ so that $v(c_0-\tilde{c}_0)>\max_{i\neq 0}\{v(\tilde{c_0}-\tilde{c}_i)\}$. This implies that $E^h(\tilde{c}_{j_0}) \subset E^h(c_0)$, so $E^h(\tilde{c}_{j_0}) = E^h(c_0)$. In particular, $\tilde{c}_{j_0}\in L$ and $\phi_{\tilde{u}}\in S$, so $\tilde{u}$ has a root $c\in F$ and the map $\tilde{c}_{j_0}\mapsto c$ embeds $E^h(c_0)$ over $E^h$ into $F$.

We can now conclude that relative separable closure of $E^h$ in $L$ embeds over $E^h$ into $F$, following the argument in \cite{4052915}. The set $\U$ together with inclusions forms a directed system. The corresponding direct limit $E'$ is the relative separable closure of $E^h$ in $L$. By the universal property of direct limit, it is enough to show that there is a set of embeddings $\{\iota_U : U \hookrightarrow_{E^h} F \ | \ U\in\U\}$ which is compatible with inclusions. For this, we consider the following inverse system. For each $U\in\U$, let $V_U$ be the set of embeddings over $E^h$ of $U$ into $F$. As shown above, this is a finite nonempty set for each $U$. Define a partial order $\le$ on $\V$ by letting $V_U\le V_{U'}$ if $U\subset U'$. Under this ordering and together with retriction maps, $\V$ forms an inverse system. Since it is an inverse system of finite nonempty sets, we have that the corresponding inverse limit is nonempty. By construction, an element of this inverse limit is a compatible system of embeddings over $E^h$, as wanted.

Since the valuation on a henselian field extends uniquely to algebraic extensions, the valuation on the image of $E'$ in $F$ induced by $v$ coincides with $w$. Hence, $\Phi$ extends to $E'$. Finally, since $L$ and $F$ are perfect, we can extend $\Phi$ to the perfect hull of $E'$, which is equal to $K$. We conclude that there is a valued field isomorphism of $(K,v)$ to a subfield of $F$ which preserves $E$. Hence, we are in the situation of Lemma \ref{lem: ake}, and we can conclude the statement.
\end{proof}

\begin{remark}
Note that by Theorem \ref{thm: ake without kaplansky}, the copy of $K$ in $F$ will indeed be the relative algebraic closure of $E$. One could get this immediately by imposing that $F\models \neg\phi_f$ for all monic $f\in E_{\ge 0}[X]$ such that $L\models \neg\phi_f$. The reason why we don't need to do this is because of relative subcompleteness; the fact that $(K,v)$ is a valued subfield of $(F,w)$ is enough to guarantee that it cannot have any proper algebraic extensions in $F$. 
\end{remark}

We now obtain Theorem \ref{thm: main ake result} and Theorem \ref{thm: mixed char ake result} as corollaries to Theorem \ref{thm: ake without kaplansky}.

\begin{proof}[Proof of Theorem \ref{thm: main ake result} and Theorem \ref{thm: mixed char ake result}]
    Let $E=\F_p(t)$, $A=\F_p[t]$,  and $\pi=t$ if $L$ has characteristic $p$. Let $E=\Q$, $A=\Z$, and $\pi=p$ if $L$ has characteristic zero. In both cases, we have that $S=S_{L,A}$. When $L$ has characteristic $p$, $v(\pi)>0$ by assumption. Otherwise, $v(\pi)>0$ since $Lv$ has characteristic $p$. We get that $v$ is the $\pi$-adic valuation on $E$, up to scaling. In particular, $(E,v)$ has rank one and $A$ is dense in the completion of $E_{\ge 0}$ with respect to $v$. Since the relative algebraic closure $(K,v)$ of $(E,v)$ in $L$ satisfies the conditions in beginning of Section \ref{sec: ake} and since $F$ satisfies the conditions of Theorem \ref{thm: ake without kaplansky} with $\Pi=\{\pi\}$, we have that $(L,v)\equiv (F,w)$ in the language $\Lval(K)$. In particular, since $\pi\in K$, we get that $(L,v)\equiv_\pi (F,w)$. When $\pi=p$, this just says that $(L,v)\equiv (F,w)$
\end{proof}

In certain cases, it is not necessary to specify the set $S$ in Theorem \ref{thm: ake without kaplansky}. For this, we fix $E$ to be $\F_p(t)$.

\begin{theorem}
\label{thm: ake with kaplansky}
Let $\F$ be a field of characteristic $p$ without any proper finite extension of degree divisible by $p$ and let $(L,v)$ be $\F((t^{1/p^\infty}))$ with the $t$-adic valuation. Let $(F,w)$ be a tame field containing $\F_p(t)$ such that $vL\equiv_{v(t)} wF$ and $Lv\equiv Fw$. Then $(L,v)\equiv(F,w)$ in the language $\L_t$.
\end{theorem}

\begin{proof}
The valuation preserving isomorphism between the copies of $\F_p(t)$ in $(L,v)$ and $(F,w)$ respectively can be extended to the perfect hull $\F_p(t)^{1/p^\infty}$ of $\F_p(t)$. Again, we can extend this to a valued fields isomorphism $\Phi$ of the henselisations $D=(\F_p(t)^{1/p^\infty})^h$ of $\F_p(t)^{1/p^\infty}$ in $L$ and $F$ respectively. We can then extend $\Phi$ to an isomorphism of valued fields with residue fields equal to $Kv$ as follows.

For an element $c\in K_{\ge 0}$, denote by $f_c(X)$ its minimal polynomial over $D$. Let $I=\{c\in K_{\ge 0} \ | \ f_c(X)\in \F_p[X]\}$ and let $\U=\{D(c) \ | \ c\in I\}$. We have that $\U$ together with inclusions is a directed system. Indeed, for two extensions $U$ and $U'$ in $\U$, their compositum is equal to $D(c, c')$, with $c$ and $c'$ being elements in $I$. Since $f_c$ and $f_{c'}$ are irreducible over $\F_p$, we have that $[\F_p(\barc):\F_p]=[U:D]$ and $[\F_p(\bar{c'}):\F_p]=[U':D]$. Let $\alpha\in Kv$ be such that $\F_p(\barc,\bar{c'})=\F_p(\alpha)$ and let $u(X)\in \F_p[X]$ be the minimal polynomial of $\alpha$. Since $u$ is separable, there is $d\in K_{\ge 0}$ such that $d$ is a lift of $\alpha$ with $u(d)=0$, as in Theorem \ref{thm: hensel}. Let $V=D(d)$. Since $[V:D]\ge [\F_p(\alpha):\F_p]$, we have that $u$ remains irreducible over $D$ and is the minimal polynomial of $d$ over $D$. This shows that $V\in\U$ and that $UU' = V$. 

Let $D'=\varinjlim \U$. For each $c\in I$, let $\phi_c$ be the $\Lval$-sentence $\exists X(f_c(X)=0)$. By definition of $I$, we have that $(L,v)\models \phi_c$ for each $c\in I$. Since $(L,v)\equiv (F,w)$ in the language $\Lval$, this implies that $(F,w)\models\phi_c$ for each $c\in I$. As in the proof of Theorem \ref{thm: ake without kaplansky}, we get that $D'$ embeds over $(\F_p(t)^{1/p^\infty})^h$ into $F$. Thus, we can extend $\Phi$ to a valued field isomorphism of the copies of $D'$ in $L$ and $F$ respectively. Note that any witness of a sentence $\phi_f$ must have valuation $0$. Indeed, let $f=\sum_{i=0}^na_iX^i$ with $a_i\in \F_p$ and suppose $v(x)\neq0$. Then $v(a_ix^i)=iv(x)$ for all non-zero $a_i$. This implies that
$$vf(x)=\min_{i: \ a_i\neq 0}\{iv(x)\}\neq \infty,$$
so $f(x)\neq 0$.

Let $\alpha\in Kv$ and let $f_{\alpha}$ be the minimal polynomial of $\alpha$ over $\Fp$. In particular, $f_\alpha$ is separable. Let $a\in K_{\ge 0}$ be a lift of $\alpha$ such that $f_\alpha(a) = 0$. As noted above, $f_\alpha$ remains irreducible over $D$. This implies that $a\in D'$, by definition of $I$. Hence, $\alpha\in D'v$, so $Kv\subset D'v$. For the converse inclusion, we just note that $D'v$ is an algebraic extension of $\Fp$. To summarise, we now have an injective field homomorphism from an algebraic subextension $(D',v)$ of $(L,v)/(\F_p(t),v)$ into $(F,w)$, preserving $\F_p(t)$, where $D'v=Kv$ and $vD'=\frac{1}{p^\infty}\Z$. We identify the image of $(D',v)$ in $(F,w)$ with $(D',v)$ itself.

By elementary equivalence in $\Lval$, both $(L,v)$ and $(F,w)$ are algebraically maximal Kaplansky fields. By Theorem \ref{thm: relative algebraic closure in maximal kaplansky}, the relative algebraic closure of $\F_p(t)$ in $L$ and $F$ are therefore algebraically maximal Kaplansky fields. Hence, $(D',v)$ is also a Kaplansky field and $(K,v)$ is the unique algebraically maximal immediate algebraic extension of $D'$, as in Lemma \ref{lem: unique algebraically maximal}. Thus, the isomorphism of the copies of $D'$ in $L$ and $F$ extends to an isomorphism of valued fields between $K$ and a subfield of $F$ and we are in the situation of Lemma \ref{lem: ake}, noting that our particular choice of $L$ satisfies the general assumptions of this section.
\end{proof}

\begin{corollary}
\label{cor: kaplansky decidable}
Let $\F$ be a field of characteristic $p$ without any proper finite extension of degree divisible by $p$. Suppose that $\F$ is decidable in the language $\Lring$. Then $\F((t\perf))$ is decidable in the language $\L_t$. In particular, if $\F$ is algebraically closed, then $\F((t\perf))$ is decidable.
\end{corollary}

\begin{proof}
By Theorem \ref{thm: ake without kaplansky}, the $\L_t$-theory of $\F((t\perf))$ is axiomatised by the $\Lval$-theory of tame fields, the $\Log(1)$-theory of $\frac{1}{p^\infty}\Z$ and the $\Lring$-theory of $\F$, which are all decidable.
\end{proof}

\begin{remark}
When $\F$ is an algebraically closed field of characteristic $p$, then $\F$ is decidable. In particular, by Corollary \ref{cor: kaplansky decidable}, there is an algorithm $\countroots\_{\F}$, which takes as input a polynomial $f(X)\in\F_p[t]$ and returns $m\in \N$, where $m$ is the number of roots of $f$ in $\F[[\tinfty]]$.
\end{remark}

\subsection{Decidability of \texorpdfstring{$\PPF$}{[Hahn field over Fp with value group the p-divisible hull of Z]}}
\label{sec: decidability of PPF}
We now turn to the question of decidability in $\L_t$ for Hahn fields of characteristic $p$ with value group $\frac{1}{p^\infty}\Z$. In this case, we can use the theory of finite automata established in Section \ref{sec: finite automata} to show that there is a recursive procedure to determine the set $S$ in Theorem \ref{thm: ake without kaplansky}. This approach was suggested by Ehud Hrushovski. We will later see how decidability of general tame Hahn fields of characteristic $p$ can be reduced to this case. To start, we need some more auxiliary algorithms.

\begin{lemma}
\label{lem: isin}
Let $q=p^n$ be a prime power. Let $\F$ be a decidable field of characteristic $p$. Then, there is an algorithm $\isin\_\F$ which on input a well-ordered DFAO $M=(Q,\Sigma_p,\delta,q_0,\F_q,\tau)$ with reachable states $F$ outputs $\true$ if $\tau(F)\subset \F$ and $\false$ otherwise.
\end{lemma}

\begin{proof}
Let $M=(Q,\Sigma_p,\delta,q_0,\F_q,\tau)$ be a well-ordered DFAO. Let $F$ be its reachable states, which are given by $\reachablestates(M)$. To start, we want to determine the minimal positive integer $k$ such that $\tau(F)\subset \F_{p^k}$. We do this by using the fact that $\tau(F)\subset \F_{p^k}$ if and only if all elements in $\tau(F)$ are roots of the polynomial $X^{p^k}-X$. As described in Section \ref{sec: finite automata}, elements in $\tau(F)$ are polynomials over $\F_p[X]$ of degree at most $n-1$. Hence, with $g(X)\in\F_p[X]$ being the irreducible polynomial of degree $n$ used to define $\F_q$ as an alphabet, determining if $\tau(F)$ is a subset of $\F_{p^k}$ amounts to checking if the sentence $\psi_k$ defined by
    $$\exists X\left(g(X)\bigwedge_{h\in \tau(F)}(h(X)^{p^k}-h(X)=0)\right)$$
holds in $\F_q$. Since any $\alpha\in\Fpbar$ witnessing $\psi_k$ will be in $\F_q$ by definition of $g$, it is enough to check if $\psi_k$ holds in $\Fpbar$, which we can do by decidability of $\Fpbar$. Hence, starting from $k=1$ and increasing $k$ until $\Fpbar\models\psi_k$ gives us the minimal positive integer $k$ such that $\tau(F)\subset \F_{p^k}$.

Now, it only remains to verify if $\F_{p^k}\subset \F$ or not. This is done using decidability of $\F$, since $\F_{p^k}\subset \F$ if and only if $\F$ has $p^k$ distinct roots to the polynomial $X^{p^k}-X$.
\end{proof}

Given $m\in\N$, we define $\Gamma_m:=\frac{1}{mp^\infty}\Z$. The following result appears as Corollary 5.4 and Corollary 5.7 in \cite{lisinski2}
\begin{lemma}
\label{lem: alg bounds}
Let $F$ be a field of characteristic $p$. Then the following hold. 
\begin{enumerate}
\item There is an algorithm $\maxram$ which takes as input a polynomial $f(X)\in\F_p(t)[X]$ and outputs a natural number $m$ not divisible by $p$ such that any root of $f$ in $\F((t^\Q))$ is already in $\F((t^{\Gamma_m}))$.
\item There is an algorithm $\maxexp$ which takes as input a polynomial $f(X)\in\F_p(t)[X]$ and outputs a natural number $m$ such that any root of $f$ in $\FbarQ$ is already in $\F_{p^m}((t^\Q))$.
\end{enumerate}
\end{lemma}

\begin{theorem}
\label{thm: ppf decidable}
Let $\F$ be a decidable perfect field of characteristic $p$. Then the Hahn field $\PPF$ is decidable in the language $\L_t$.
\end{theorem}

\begin{proof}
By Theorem \ref{thm: ake without kaplansky}, the $\L_t$-theory of $\PPF$ is axiomatised by the $\Lval$-axioms for tame fields of characteristic $p$, the $\Log(v(t))$-theory of $\frac{1}{p^\infty}\Z$, the $\Lring$-theory of $\F$, and the set $S$ of one variable positive existential $\L_t$-sentences satisfied by $\PPF$. The $\Log(v(t))$-theory of $\frac{1}{p^\infty}\Z$ is decidable by Lemma \ref{lem: ord grps decidable}, and the $\Lring$-theory of $\F$ is decidable by assumption. In particular, they are both recursively enumerable. Hence, it is enough to show that there is an algorithm $\enumeratealg\_S$ which enumerates $S$ to conclude that the $\L_t$-theory of $\PPF$ is recursively enumerable. We will show something slightly stronger, namely that $S$ is decidable. For this, we create an algorithm $\decidealg\_S$ as input a sentence $\phi_f$ with $f\in\F_p[t]$, and outputs $\true$ if $\phi_f\in S$ and $\false$ otherwise. This algorithm is outlined as follows.
\begin{enumerate}
    \item Let $m=\countroots\_\Fbar(f)$, i.e. the number of unique roots of $f$ in $\PRFbar$, and let $n=\maxexp(f)$. By Lemma \ref{lem: rel alg closure in Hahn field}, any root of $f$ in $\PRFbar$ is contained in $\Fpbar((\tinfty))$. So by the second item of Lemma \ref{lem: alg bounds}, all such roots lie in $\F_q[[\tinfty]]$, where $q=p^n$.
    \item Find DFAOs $M_1,\ldots,M_m$ such that $\{\Pow(M_i) \ | \ 1\le i\le m\}$ is the set of all roots of $f$ in $\F_q[[\tinfty]]$ by going through well-ordered DFAOs with output alphabet $\F_q$ and check if they represent roots to $f$ using $\isrootalg$. When finding a DFAO $M$ such that $\isrootalg(f,M)$ returns $\true$, we verify that $\Pow(M)$ is not equal to any of the previous found roots using $\equals$.
    \item Check if any of these roots are in $\PRF$, using the algorithm $\isin\_\F$ from Lemma \ref{lem: isin}.
\end{enumerate}

In pseudocode, the algorithm is given as follows.
\begin{algorithm}[H]
\begin{algorithmic}
\caption*{$\decidealg\_S(f)$}
\State $m\gets \countroots\_\Fbar(f)$
\State $n\gets \maxexp(f)$
\State $q\gets p^n$
\State $R\gets \emptyset$
\State $k\gets 0$
\While {$\lvert R\rvert < m$}
    \State $M\gets \listwellord\_\F_q(k)$
    \If{$\isrootalg(f,M)$}
        \State $\IsRoot \gets \true$
        \For {$N\in R$}
            \If {$\equals(M,N)$}
                \State $\IsRoot \gets \false$
            \EndIf
        \EndFor
    \EndIf
    \If{$\IsRoot$}
        \State $R\gets R\cup\{M\}$
    \EndIf
    \State $n\gets n+1$
\EndWhile
\For{$M\in R$}
    \If{$\isin\_\F(M)$}
        \State \Return $\true$
    \EndIf
\EndFor
\State \Return $\false$
\end{algorithmic}
\end{algorithm}
We conclude that the set $S$ is decidable, so $\PPF$ is decidable in the language $\L_t$.
\end{proof}

\begin{remark}
An algorithm which enumerates $S$ can be constructed by listing the $j$ first polynomials $f_1,\ldots, f_j$ over $\F_p[t]$ and the $j$ first well-ordered DFAs $M_1,\ldots, M_j$ with output alphabet $\F_{q_j}$, where $q_j=p^{m_j}$ is a subfield of $\F$ such that any root of $f_1,\ldots,f_j$ in $\PPF$ is already in $\F_q((t\perf))$, and such that $\F_{q_j}\subset \F_{q_{j+1}}$. If $f_k(\Pow(M_\ell))$ for some $k$ and some $\ell$ less than or equal to $j$, and if $f_k$ is not already equal to $\enumeratealg\_S(i)$ for some $i<n$, then $\enumeratealg\_S$ returns $f_k$. If this is not the case, we increase $j$ and repeat. The condition that $\F_{q_j}\subset\F_{q_{j+1}}$ is met using $\maxexp$ and it ensures that we will list all DFAOs that represent a root in $\PPF$ of some polynomial $f\in\F_p[t][X]$. The reason why why instead show that $S$ is decidable is to avoid the technicalities of comparing automata with different output alphabets.
\end{remark}

\subsection{Decidability of general positive characteristic tame Hahn fields}

We now turn to the question of decidability when there is ramification at primes different from the characteristic. As above, given $m\in\N$ we write 
$$\Gamma_m=\frac{1}{mp^\infty}\Z.$$
We start with the following observation.

\begin{corollary}
\label{cor: simple ramification}
Let $\F$ be a decidable perfect field of characteristic $p$.
Then, for any $m\in\N$, we have that $\F((t^{\Gamma_m}))$ is decidable in $\L_t$.
\end{corollary}

\begin{proof}
This follows from Theorem \ref{thm: ppf decidable}, with $t^{1/m}$ in place of $t$.
\end{proof}

\begin{theorem}
\label{thm: decidable containing hahn field}
Let $(L,v)$ be a tame field containing $\F_p(t)$. Suppose that $Lv$ and $vL$ are decidable in $\Lring$ and $\Log(v(t)$ respectively. Let $\F$ be a perfect decidable subfield of $Lv$ containing the relative algebraic closure of $\F_p$ in $Lv$ and let $\Gamma$ be the relative divisible hull of $\langle v(t)\rangle$ in $vL$. Suppose that $(\F((t^\Gamma)),v_t)$ is a valued subfield of $(L,v)$. Then, $(L,v)$ is decidable in $\L_t$
\end{theorem}

\begin{proof}
We will use Theorem \ref{thm: ake without kaplansky}. For this, we first note that the relative algebraic closure $K$ of $\F_p(t)$ in $L$ is contained in $\F((t^\Gamma))$. Indeed, let $K'$ be the relative algebraic closure of $\F_p(t)$ in $\F((t^\Gamma))$. Then $(K',v)$ is tame by Lemma \ref{lem: hahn tame rel condition}. In particular, $K'$ is algebraically maximal. Since $Kv$ is contained in $\F$, we have that $K'v=Kv$. Since $\Gamma$ is the relative divisible hull of $\langle v(t)\rangle$, which is equal to $vK$, we also have that $vK'=vK$. Hence, $K/K'$ is an immediate algebraic extension, so $K=K'$. Since $vK=\Gamma$, we have that $vL/vK$ is torsion free. We are therefore in the situation of Theorem \ref{thm: ake without kaplansky}, and it is again enough to show that there is a decision procedure for the set $S$ defined before Theorem \ref{thm: ake without kaplansky}. 

Now, let $f(X)\in\F_p[t][X]$ be monic. Let $m=\maxram(f)$, as in the first item of Lemma \ref{lem: alg bounds}. Then any root of $f$ in $\F[[t^\Q]]$ is already in $\F[[t^{\Gamma_m}]]$, where $\Gamma_m=\frac{1}{mp^\infty}\Z$ as above. Let $U$ be the set of factors of $m$. Let $V$ be the set of natural numbers $n$ such that $n$ is not divisible by $p$ and such that $\frac{u}{n}\in \Gamma$ for some $u\in\Z$, where $\langle v(t)\rangle$ is identified with $\Z$ in $\Gamma$. Note that $\frac{u}{n}\in \Gamma$ if and only if $\frac{u}{np^e}\in \Gamma$ for any $e\in \N$, since $\Gamma$ is $p$-divisible. Furthermore, we have that $\frac{u}{n}\in V$ if and only if $\frac{u'}{n}+v\in \Gamma$ for some $u'\in \{1,\ldots, n-1\}$ and $v\in\Z$. Since $\Z$ is a subgroup of $\Gamma$, this holds if and only if $\frac{u'}{n}\in \Gamma$. Using decidability of $vL$, we get that there is an algorithm which on input $f$ outputs the set
$$U\cap V=\{n_1,\ldots,n_\ell\}.$$
Indeed, since $m$ is not divisible by $p$, this algorithm outputs exactly the elements $n\in U$ for which $vL\models\bigvee_{0<u<n} \exists X(nX=u)$.

Define
$$m':=\prod_{i=1}^\ell n_i.$$
Since $\F[[t^\Gamma]]\subset\F[[t^\Q]]$, any root of $f$ in $\F[[t^\Gamma]]$ is already in $\F[[t^{\Gamma_m}]]$. 

Let $x\in \F[[t^{\Gamma_m}]]\cap\F[[t^\Gamma]]$. Then, since the support of $x$ is in $\Gamma_m$, we can write
$$x=\sum_{i\in I}a_it^{\frac{m_i}{n_ip^i}}$$
where $I\subset \N$, $m_i\in\N$, and $n_i\in U$. On the other hand, since the support of $x$ is in $\Gamma$ and since no $n_i$ is divisible by $p$, we have that each $n_i$ is in $V$. Hence, $x\in\F[[t^{\Gamma_{m'}}]]$. Conversely, if $x\in\F[[t^{\Gamma_{m'}}]]$, we write
$$x=\sum_{i\in I}a_it^{\frac{m_i}{n_ip^i}}$$
where each $n_i\in U\cap V$. Since each $n_i$ divides $m$, we get that $x\in \F[[t^{\Gamma_m}]]$ and since $n_i\in V$, we get that $x\in\F[[t^\Gamma]]$. We conclude that $\F[[t^{\Gamma_m}]]\cap\F[[t^\Gamma]]=\F[[t^{\Gamma_{m'}}]]$. We can thus use the decision procedure from Corollary \ref{cor: simple ramification} for $\F((t^{\Gamma_{m'}}))$ to determine if $f$ has a root in $\F[[t^\Gamma]]$ or not. Since any root $x\in L$ of $f$ is already in $\F((t^\Gamma))$ as noted above, we can now use this same decision procedure to determine if $f$ has a root in $K$ or not, and we are done.
\end{proof}

We now get Theorem \ref{thm: main decidability result} as an immediate consequence of Theorem \ref{thm: decidable containing hahn field}, since $\F((t^G))$ is a subfield of $\F((t^\Gamma))$, with $G$ being the relative divisible hull of $\langle v(t)\rangle$ in $\Gamma$. 

\section*{Acknowledgment}
I wish to thank Jochen Koenigsmann for introducing me to this research question and for continued guidance and proofreading throughout the work. I also wish to thank Ehud Hrushovski for insightful comments, in particular for suggesting to use the work on finite automata. Finally, I wish to thank Konstantinos Kartas for proofreading and many fruitful discussions.

\bibliographystyle{alpha}
\bibliography{biblio2}

\Addresses

\appendix

\section{More on finite automata}
\label{appendix}

We denote by $\rev:\Sigma^*\to\Sigma^*$ the function sending $s_1\cdots s_n$ to $s_n\cdots s_1$. The following appears as Theorem 4.3.3 in \cite{MR1997038}.
\begin{theorem}
\label{thm: reverse DFAO}
There is an algorithm $\REV$ which takes as input a DFAO $M$ and returns a DFAO $M'$ such that $f_{M'}=f_M\circ\rev$.
\end{theorem}

We fix some notation for the following lemma. Let $n$ be a positive integer. Given an NFA $M = (Q, \Sigma, \delta, q_0, F)$ and a string $w\in \Sigma^*$, denote by $a_M(w)$ the number of accepting paths for $w$ in $M$. Let
\begin{align*}
    g_M : \   & \Sigma^* \to \Z/n\Z \\
            & w\mapsto [a_M(w)].
\end{align*}

\begin{lemma}[{{\cite[Lemma 2.2.2]{MR2289431}}}]
\label{lem: counting paths}
There is an algorithm $\accpaths$ which takes as input an $NFA$ $M$ and outputs a DFAO $M'$ such that $g_M=f_{M'}$, with $g_M$ being as defined above.
\end{lemma}

It is often convenient to consider finite automata as edge-labeled directed graphs, i.e. as directed graph $(V,E)$ together with a labeling function $\ell$ from $E$ to some set $S$. In particular, this will be useful to verify effectiveness of results in this section.

\begin{remark}
There is an algorithm $\subgraph$ which takes as input a finite edge-labeled directed graph $G=(V,E,\ell)$ and a subset $U\subset V$ and returns the edge set $E_U$ of the induced subgraph $G[U]$ together with the labelling function of $G$ restricted to $E_U$. It is defined as follows.
\begin{algorithm}[H]
\caption*{$\subgraph(G=(V,E,\ell), U\subset V)$}
\label{alg:subgraph}
\begin{algorithmic}[1]
\State {$E_U\gets\emptyset$}
\State {$\ell_U\gets \emptyset$}
\For {$(u,v)\in U^2$}    
    \If {$(u,v)\in E$}
        \State $E_U \gets E_U\cup \{(u,v)\}$
        \For{$s\in \ell(E)$}
            \If{$\ell(u,v)=s$}
                \State $\ell_U\gets \ell_U\cup\{((u,v),s)\}$
            \EndIf    
        \EndFor        
    \EndIf    
\EndFor
\State \Return $(E_U,\ell_U)$
\end{algorithmic}
\end{algorithm}
\end{remark}

\begin{definition}
\label{def: transition graph}
Let $M=(Q,\Sigma,\delta,q_0,\Delta,\tau)$ be a DFAO or an NFA. The \textbf{transition graph} of $M$ is the edge-labeled directed graph on the vertex set $Q$, with an edge from $q\in Q$ to $q'\in Q$ labeled by $s\in \Sigma$ if $\delta(q,s)=q'$.
\end{definition}

When illustrating transition graphs, we will use double circles around a state to illustrate that it is an accepting state. Furthermore, we will omit all edges that do not lead to accepting paths, and all states only reached by such paths. For example, the following is an illustration of a DFA which only accepts the string $s_1\cdots s_n$.
  

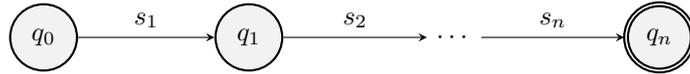
\begin{center}
\begin{tikzpicture}
\node[state] (q0) {$q_0$};
\node[state, right of=q0] (q1) {$q_1$};
\node[right of=q1] (dots) {$\cdots$};
\node[state, accepting, right of=dots] (qn) {$q_n$};
\draw (q0) edge[above] node{$s_1$} (q1);
\draw (q1) edge[above] node{$s_2$} (dots);
\draw (dots) edge[above] node{$s_n$} (qn);
\end{tikzpicture}
\end{center}

\begin{remark}
The following defines an algorithm which takes as input a DFAO or an NFA $M=(Q,\Sigma,\delta,q_0,\Delta,\tau)$ and returns the edges $E$ and the labelling $\ell$ of the transition graph of $M$.
\begin{algorithm}[H]
\caption*{$\transition(M=(Q,\Sigma,\delta,q_0,\Delta,\tau))$}\label{alg:transition}
\begin{algorithmic}[1]
\State $E \gets \emptyset$
\State $\ell \gets \emptyset$
\For{$(q,q',s)\in Q^2\times\Sigma$} 
    \If{$\delta(q,s)=q'$}
        \State $E \gets E \cup \{(q,q')\}$
        \State $\ell \gets \ell \cup \{((q,q'),s)\}$
    \EndIf
\EndFor
\State \Return $(E,\ell)$
\end{algorithmic}
\end{algorithm}
\end{remark}

\begin{remark}
\label{rem: well-formed bound}
To verify that $M=(Q,\Sigma_p,\delta,q_0,\Delta,\tau)$ is well-formed, it is enough to consider strings of length at most $m=(3p+1)\lvert Q\rvert+2$. Indeed, suppose that $n>m$ and that $M$ accepts a string $w=s_1\cdots s_n$ which is not the valid base $p$-expansion of any $a\in\frac{1}{p^\infty}\N$. By definition, this means that $w$ satisfies one of the following items.
\begin{enumerate}
    \item $s_1=0$;
    \item $s_n=0$;
    \item no $s_i$ is equal to the radix point;
    \item $s_i$ and $s_j$ are both equal to the radix point with $i\neq j$.
\end{enumerate}

Suppose that the first or third item holds. Since $n>\lvert Q\rvert+1$, we have that the sequence $\delta^*(q_0,s_1s_2)$, $\delta^*(q_0,s_1s_2s_3)$, $\ldots$, $\delta^*(q_0,s_1\cdots s_n)$ contains two identical states. Let $i,j\in\{2,\ldots,n\}$ be distinct such that $\delta^*(q_0,s_1\cdots s_i)=\delta^*(q_0,s_1\cdots s_j)$. Without loss of generality, assume that $i<j$ and write $s_{\hat{i}}=s_1\cdots s_i s_{j+1}\cdots s_n$. Then $\delta^*(q_0,s_{\hat{i}})=\delta^*(q_0,w)$. In particular, $M$ accepts the string $s_1\cdots s_i s_{j+1}\cdots s_n$ which is of length strictly less than $n$ beginning with $s_1$. By induction, we get that $M$ accepts a string $w'$ of length $m$, with elements being a subset of $\{s_1,\ldots, s_n\}$, which also begins with $s_1$. By construction, if $w$ satisfies the first item then so does $w'$, and if $w$ satisfies the third item then so does $w'$. Hence $w'$ is not a valid base $p$ expansion. By the same argument, fixing $s_n=0$ instead of $s_1$, we get that $M$ accepts a string of length $m$ which is not a valid base $p$ expansion if the second item hold.

Suppose now that the fourth item holds. Let $e_1,\ldots,e_n$ be the sequence of edges corresponding to the sequence of connected vertices $$q_0,\delta(q_0,s_1),\delta^*(q_0,s_1s_2),\ldots,\delta^*(q_0,s_1\cdots s_n).$$ Let $i$ and $j$ be such that $s_i$ is the first radix point of $w$ and $s_{i+j}$ is the second radix point of $w$. We can assume that the subsequences $(e_k)_{1\le k < i}$, $(e_k)_{i<k < i+j}$ and $(e_k)_{i+j<k\le n}$ all separately only contain distinct edges. Indeed, if $e_k=e_\ell$ for some $k<\ell$ with
$$(k,\ell)\in \{1,\ldots,i-1\}^2 \cup \{i+1,\ldots,i+j-1\}^2 \cup \{i+1+1,\ldots,n\}^2$$
then we repeatedly replace $w$ with the string $s_1\cdots s_ks_{\ell+1}\cdots s_n$, which by construction also contains two radix points and is accepted by $M$. We will see that this string must have length less than or equal to $m$, thus showing that $M$ accepts a string of length less than or equal to $m$ which is not a valid base $p$ expansion.

Since $\Sigma_p$ contains $p+1$ elements, the total number of edges of the transition graph of $M$ is $(p+1)\lvert Q\rvert$, and the total number of edges that are not labeled by the radix point is $p\lvert Q\rvert$. Since $(e_k)_{1\le k <i}$ only contains distinct vertices, none labeled by the radix point, we have that $i\le p\lvert Q\rvert+1$. Similarly, we get that $j\le p\lvert Q\rvert+1$. By the assumption that $(e_k)_{i+j< k\le n}$ only contains distinct edges, the length of this sequence must be bounded above by the total number of edges in the transition graph. Hence, $n-i-j\le (p+1)\lvert Q\rvert$. Combining the obtained inequalities gives
$$n\le (3p+1)\lvert Q\rvert +2,$$
and we are done.
\end{remark}

\begin{remark}
\label{rem: well-formed effective}
By the bound obtained in Remark \ref{rem: well-formed bound}, we conclude that there is an algorithm which takes as input a DFAO $M=(Q,\Sigma_p,\delta,q_0,\Delta,\tau)$ and returns $\true$ if $M$ is well-formed, and $\false$ otherwise. We denote this algorithm by $\wellformed$. It is defined as follows.
\begin{algorithm}[H]
\caption*{$\wellformed(Q,\Sigma_p,\delta,q_0,\Delta,\tau)$}\label{alg:wellformed}
\begin{algorithmic}[1]
\For{$s=s_1\cdots s_n\in \Sigma_p^*$ with $n\le(3p+1)\lvert Q\rvert+2$} 
    \LineComment{Verify that no string starting or ending with $0$ is accepted.}
    \If{$s_1=0 \orbf s_n=0$}
        \If{$f_M(s)\neq 0$}
            \State \Return $\false$
        \EndIf
    \EndIf
    \LineComment{Verify that every accepted string has exactly one radix point.}
    \If{$s_i\neq .$ for all $i\in\{1,\ldots,n\} \orbf s_i=s_j=.$ for some $i\neq j$}
        \If{$f_M(s)\neq 0$}
            \State \Return $\false$
        \EndIf
    \EndIf
\EndFor
\State \Return $\true$

\end{algorithmic}
\end{algorithm}

\end{remark}

To verify that a DFAO is well-ordered, we need a bit more.

\begin{definition}
\label{def: relevant state}
A state $q\in Q$ is \textbf{relevant} if there exists an accepting state reachable from $q$. If $q$ is not relevant, we say that it is \textbf{irrelevant}.
\end{definition}

\begin{remark}
\label{rem: relevant states}
There is an algorithm which takes as input a DFAO $M$ together with a state $q$ of $M$ and returns the relevant states reachable from $q$. We denote it by $\relevantstates$. It is defined as follows. As in Remark \ref{rem: bounded strings to reach}, we only need to consider strings of length less than or equal to $\lvert Q\rvert$. 
\begin{algorithm}[H]
\caption*{$\relevantstates(M=(Q,\Sigma,\delta,q_0,\Delta,\tau),q)$}\label{alg:relevant}
\begin{algorithmic}
\State $Q_q\gets\emptyset$
\For{$s=s_1\cdots s_n\in\Sigma^*$ with $n\le \lvert Q\rvert$}
    \If{$\tau(\delta^*(q,s))\neq 0$}
        \State $Q_q\gets Q_q\cup\{\delta^*(q,s_1\cdots s_i) \ | \ \text{for } 0\le i \le n\}$
    \EndIf
\EndFor
\State \Return $Q_q$
\end{algorithmic}
\end{algorithm}
We write $\relevantstates(M)$ in place of $\relevantstates(M,q_0)$. Note that if $M$ is minimal, then all states are reachable from $q_0$, so all relevant states of $M$ are in this case given by $\relevantstates(M)$. 
\end{remark}

\begin{definition}
\label{def: postradix}
Let $M$ be a DFAO with input alphabet $\Sigma_p$. We say a state $q \in Q$ is \textbf{preradix} (resp. \textbf{postradix}) if there exists a valid base $p$ expansion $s = s_1 \cdots s_n$ accepted by $M$ with $s_k$ equal to the radix point such that $q = \delta^*(q_0, s_1\cdots s_i)$ for some $i < k$ (resp. for some $i \ge k$).
\end{definition}

\begin{remark}
If $M$ has input alphabet $\Sigma_p$ and is well-formed, then no accepted state can be both preradix and postradix. Indeed, suppose that there is an accepted state $q$ which is both preradix and postradix. Then, there are valid base $p$ expansions $s_1\cdots s_m$ and $s'_1\cdots s'_n$ accepted by $M$ with radix points $s_k$ and $s'_\ell$ respectively such that
$$q = \delta^*(q_0, s_1\cdots s_i)=\delta^*(q_0, s'_1\cdots s'_j),$$
with $i<k$ and $j\ge \ell$. We then have that
$$\delta^*(q_0,s_1\cdots s_m)=\delta^*(q_0, s'_1\cdots s'_j s_{i+1}\cdots s_m)$$
is an accepted state reached by a string with two radix points, which is a contradiction. 
\end{remark}

\begin{remark}
\label{rem: postradix alg}
By definition, any postradix state $q$ is relevant. Furthermore, any relevant state is postradix if and only if there is a string $w=s_1\cdots s_n$ not containing any radix point such that $\delta^*(q,w)$ is an accepting state. This, combined with Remark \ref{rem: bounded strings to reach} shows that the following defines an algorithm which takes as input a well-formed DFAO $M$ and outputs the postradix states of $M$. A similar algorithm exists for preradix states, though we will not need it.
\begin{algorithm}[H]
\caption*{$\postradix(M=(Q,\Sigma_p,\delta,q_0,\Delta,\tau))$}\label{alg:postradix}
\begin{algorithmic}
\State $\text{PostradixStates} \gets \emptyset$
\State $\RelevantStates \gets \relevantstates(M)$
\For {$q\in \RelevantStates$}
    \For {$w=s_1\cdots s_n\in\{0,\ldots,p-1\}^*$ with $n\le \lvert Q\rvert$}
        \If{$\tau(\delta^*(q,w))\neq 0$} 
            \State $\text{PostradixStates} \gets \text{PostradixStates} \cup \{q\}$
        \EndIf
    \EndFor
\EndFor
\State \Return $\text{PostradixStates}$
\end{algorithmic}
\end{algorithm}
\end{remark}

We will now consider a characterisation of transition graphs which will allow us to tell if a DFAO is well-ordered.

\begin{definition}
\label{def: saguaro}
Let $G=(V,E)$ be a directed graph and let $v\in V$. We say that $G$ is a \textbf{rooted saguaro}, and that $v$ is a \textbf{root} of $G$, if the following hold.
\begin{enumerate}
    \item Each vertex of $G$ lies on at most one cycle, up to permutation of the vertices by rotation.\footnote{With cycle, we mean a closed directed path with no repeated vertices apart from the first and last. In \cite{MR2289431}, the term \textit{minimal cycle} is used for what we call a cycle.}
    \item There exists directed paths from $v$ to each vertex of $G$.
\end{enumerate}
An edge of a rooted saguaro is \textbf{cyclic} if it lies on a cycle and \textbf{acyclic} otherwise.
\end{definition}

\begin{remark}
\label{rem: saguaro effective}
There is an algorithm which takes as input a finite directed graph $G=(V,E)$ and an element $v\in V$ and outputs $\true$ if $G$ is a rooted saguaro with root $v$, and $\false$ otherwise. We denote this algorithm by $\saguaro$. To see that such an algorithm indeed exists, we first let $\cycles$ be an algorithm which takes as input a directed graph and return the set of finitely many cycles of $G$ (for example by going through all possible paths of length at most $\lvert V\rvert+1$). From this, we can check that each vertex only appears in one such cycle, up to rotation of the vertices. For the second item, it is enough to go through directed paths of length at most $\lvert V\rvert-1$ to verify that each vertex of $G$ is in such a path. Indeed, if there is a path $(v_1,v_2), \ (v_2,v_3), \ldots, \ (v_n,v_{n+1})$ from $v=v_1$ to $w=v_{n+1}$ with $n\ge \lvert V\rvert$, then $v_i=v_j$ for some $i\neq j$. This gives a path $(v_1,v2), \ldots, \ (v_i,v_{j+1}), \ldots \ (v_n,v_{n+1})$ of length strictly less than $n$.
\end{remark}

\begin{definition}
\label{def: p-labeling}
Let $G=(V,E)$ be a directed graph. A \textbf{proper $p$-labeling} of $G$ is a function $\ell:E\to \{0,\ldots,p-1\}$ with  the following properties hold for all $v,w,x \in V$.
\begin{enumerate}
    \item If $w\neq x$ and $(vw,vx)\in E^2$, then $\ell(vw)\neq \ell(vx)$.
    \item If $vw\in E$ lies on a cycle and $vx\in E$ does not lie on a cycle, then $\ell(vw)>\ell(vx)$
\end{enumerate}
\end{definition}

\begin{remark}
There is an algorithm $\plabeling$ which takes as input a directed graph $G=(V,E)$ together with a function $\ell: E\to\{0,\ldots,p-1\}$ and returns $\true$ if $\ell$ is a proper $p$-labeling of $E$ and $\false$ otherwise. It is defined as follows.
\begin{algorithm}[H]
\caption*{$\plabeling(V,E,\ell)$}
\label{alg:plabeling}
\begin{algorithmic}[1]
\State {$\text{Cycles} \gets \cycles(V,E)$}
\For {$(v,w,x)\in V^3$}
    \If{$w\neq x \andbf vw\in E \andbf vx\in E$}
        \If{$\ell(v,w)=\ell(v,x)$}
            \State \Return $\false$
        \EndIf
    \EndIf
    \If{$vw\in \text{Cycles} \andbf vx\in E\setminus \text{Cycles}$}
        \If {$\ell(vw)\le\ell(vx)$}
            \State \Return $\false$
        \EndIf
    \EndIf
\EndFor
\State \Return $\true$
\end{algorithmic}
\end{algorithm}
\end{remark}

\begin{theorem}[{{\cite[Theorem 7.1.6]{MR2289431}}}]
\label{thm: well-ordered DFA}
Let $M$ be a DFA with input alphabet $\Sigma_p$ and suppose that $M$ is minimal and well-formed. For any state $q$ of $M$, let $G_q$ be the subgraph of the transition graph consisting of relevant states reachable from $q$. Then $M$ is well-ordered if and only if for each relevant postradix state $q$, we have that $G_q$ is a rooted saguaro with
root $q$, equipped with a proper $p$-labeling.
\end{theorem}

We will now fix some notation to allow us to talk go between elements in $\PRFq$ and their corresponding automata. If $a\in \F_q$, we denote by $aM$ the DFAO given from $M$ by replacing the output function $\tau$ with $a\tau$. By definition, this implies that $\Pow(aM)=a\Pow(M)$. We write $-M$ instead of $-1M$. Note that $\Pow(M)=0$ if and only if $M$ has no reachable accepting states. When $M$ is minimal, this is equivalent to $\relevantstates(M)=\emptyset$.

To define a DFA corresponing to $1\in \PRFq$, we need a DFA $$\mathbbm{1}=(Q,\Sigma_p,\delta,q_0,\{0,1\},\tau\})$$
which accepts only the string $s(0)$, i.e. the string only consisting of a single radix point. More precisely, this is defined as follows.
\begin{itemize}
    \item $Q=\{q_0,q_1,q_2\}$, where $q_i=i$.
    \item $\delta(q_0,s)=q_1$ for all $s\in\{0,\ldots,p-1\}$.
    \item $\delta(q_0,s(0))=q_2$.
    \item $\delta(q_i,s)=q_1$ for $i\in \{1,2\}$ and for all $s\in \Sigma_p$.
    \item $\tau(q_i)=0$ for $i\in\{0,1\}$ and $\tau(q_2)=1$.
\end{itemize}
With this definition, $\mathbbm{1}$ is minimal and well-ordered, and $\Pow(\mathbbm{1}=1$.

\begin{remark}
\label{rem: linearisation}
We can write any $x\in\PRFq$ as a linear combination over $\F_q$. In particular, if $M$ is a  well-ordered DFAO $M$, we get 
$$\Pow(M)=\sum_{i=1}^{q-1} a_i\sum_{r\in S_i}t^r\in\PRFq$$
where $a_i\in\F_q^*$ and $a_i\neq a_j$ for $i\neq j$, and where $S_i$ is a well-ordered, possibly empty, subset of $\frac{1}{p^\infty}\Z$. With this in mind, there is an algorithm $\lincomb$ which takes as input a DFAO $M$ with output alphabet $\F_q$ and returns a set of tuples $\{(a_i,M_i)\ | \ i\in \{1,\ldots,q-1\}\}$ where the $a_i\in\F_q$ are all distinct and the $M_i$ are minimal DFAs with input alphabet $\Sigma_p$ such that, if $M$ is well-ordered, we have
\begin{align}
\label{eqn: pow_linearisation}
    \Pow(M)=\sum_{i=1}^{q-1} a_i\Pow(M_i).
\end{align}
If $M=(Q,\Sigma_p,\delta,q_0,\F_q,\tau)$, let $M'_i$ be the DFA given by $(Q,\Sigma_p,\delta,q_0,\F_q,\tau_i)$, where $\tau_i(q)=1$ if $\tau(q)=a_i$ and $\tau_i(q)=0$ otherwise. Then $(M'_i)_{1\le i \le n}$ satisfies (\ref{eqn: pow_linearisation}). Taking $M_i$ to be $\mindfa(M'_i)$ gives the desired result. By construction, we also have that $M$ is well-ordered if and only if all of the $M_i$ are well-ordered. Indeed, if $M_i$ is not well-ordered, then there is an infinite descending sequence $(v_j)_{j\in \N}$ of elements in $\frac{1}{p^\infty}\N$ such that $M_i$ accepts the valid base $p$-expansion of each $v_j$. By definition of $\tau_i$, this implies that $M$ also accepts the valid base $p$-expansions of each $s_j$, so $M$ is not well-ordered. Conversely, suppose that $M$ is not well-ordered, accepting the valid base $p$ expansions of an infinite descending sequence $(v_j)_{j\in \N}$. By definition, this means that $f_M(s(v_j))\neq 0$ for all $j$. By the pigeonhole principle, there is an $i\in \{1,\ldots,q-1\}$ and an infinite subsequence $(v'_j)_{j\in \N}$ such that $f_M(s(v'_j))=a_i$. By definition of $\tau_i$, this implies that $M_i$ accepts the valid base $p$ expansions of $(v'_j)_{j\in \N}$ as well, so $M_i$ is not well-ordered.
\end{remark}

\begin{remark}
\label{rem: equals}
Using $\lincomb$, we see that there is an algorithm $\equals$ which takes as input two well-ordered DFAOs $M$ and $N$ with output alphabet $\F_q$, returning $\true$ if $\Pow(M)=\Pow(N)$ and $\false$ otherwise. Since $\mindfa$ gives a unique minimal DFA, up to renaming the state, the algorithm $\equals$ amounts to checking that, for every $a\in\F_q$, if $(a,M')$ and $(a,N')$ are tuples in $\lincomb(M)$ and $\lincomb(N)$, then $M'$ and $N'$ are equal, up to renaming the states.
\end{remark}

\begin{remark}
\label{rem: well-ordered DFAO decidable}
From Theorem \ref{thm: well-ordered DFA} and Remark \ref{rem: linearisation}, we conclude that there is an algorithm which takes as input a DFAO $M$ with input alphabet $\Sigma_p$ and returns $\true$ if $M$ is well-ordered, and $\false$ otherwise. We denote this algorithm by $\wellord$. It is defined as follows.

\begin{algorithm}[H]
\caption*{$\wellord(M=(Q,\Sigma_p,\delta,q_0,\Delta,\tau))$}
\label{alg:wellord}
\begin{algorithmic}[1]
\For{$(a,N)\in\lincomb(M)$}
    \LineComment{Verify that $N$ satisfies the properties of Theorem \ref{thm: well-ordered DFA}.} 
    \If {$\notbf \ \wellformed(N)$}
        \State \Return $\false$
    \EndIf
    \State $(E,\ell) \gets \transition (N)$
    \State $\RelevantStates \gets \relevantstates(N,q_0)$
    \State $\PostradixStates \gets \postradix(N)$
    \For {$q\in\RelevantStates \cap \PostradixStates$}
        \State {$Q_q\gets\relevantstates(N,q)$}
        \State {$(E_q,\ell_q)\gets\subgraph((Q,E,\ell),Q_q)$}
        \If {$\notbf \saguaro((Q_q,E_q)$}
            \State \Return $\false$
        \EndIf
        \If{$\notbf \plabeling(Q_q,E_q,\ell_q)$}
            \State \Return $\false$
        \EndIf
    \EndFor
\EndFor
\State \Return $\true$
\end{algorithmic}
\end{algorithm}
\end{remark}

\begin{remark}
\label{rem: list well-ordered}
We can now define the algorithm $\listwellord\_\F_q$ in Lemma \ref{lem: list well ord}. On input $n$, this algorithm returns the $n$:th DFAO in $\calD(\Sigma_p,\F_q)$ which is well-formed, which is verified using $\wellord$.
\end{remark}

In Section \ref{sec: decidability of PPF}, we will use Theorem \ref{thm: algebraic iff p-automatic} to determine which monic one variable polynomials over $\F_p[t]$ have roots in certain Hahn fields. To this end, we will 
need the following two lemmas describing an effective procedure for arithmetic of DFAOs representing elements in $\PRFq$. They appear implicitly in \cite{MR2289431} as Lemma 7.2.1 and Lemma 7.2.2 respectively. Again, the proofs are entirely due to Kedlaya and are included to emphasise effectiveness.

\begin{lemma}
\label{lem: DFAO addition}
There is an algorithm $\add$ which takes as input two well-ordered DFAOs $M$ and $M'$, both having input alphabet $\Sigma_p$ and output set $\F_q$, and outputs a well-ordered DFAO $N$ with the same input alphabet and output set such that $\Pow(N)=\Pow(M)+\Pow(M')$.
\end{lemma}

\begin{proof}
Let $M=(Q,\Sigma_p,\delta,q_0,\F_q,\tau)$ and $M'=(Q',\Sigma_p,\delta',q'_0,\F_q,\tau')$. Define
$$N=(Q\times Q', \Sigma_p,\tilde{\delta},(q_0,q'_0), \F_q,\tilde{\tau})$$
where $\tilde{\delta}((q,q'),s)=\delta(q,s)\times\delta'(q',s)$ and $\tilde{\tau}(q,q')=\tau(q)+\tau'(q')$. By construction, $N$ is a DFAO. Let $w\in\Sigma^*$ be a valid base $p$ expansion. We have that
\begin{align*}
    f_N(w) & = \tilde{\tau}(\tilde{\delta}^*((q_0,q'_0),w)) \\
    & = \tilde{\tau}(\delta^*(q_0,w),\delta'^*(q'_0,w)) \\
    & = \tau(\delta^*(q_0,w))+\tau'(\delta^*(q'_0,w)) \\
    & = f_M(w)+f_{M'}(w).
\end{align*}
Hence, $f_N=f_{M}+f_{M'}$ and so $\Pow(N)=\Pow(M)+\Pow(M')$. We define $\add$ to return $N$ on input $M$.
\end{proof}

In the following result, we will consider reversed base $p$ expansions. This will not just be a matter of reversing valid base $p$-expansions, but we will also allow for leading and trailing zeroes. This is to make sense of base $p$ addition with carries. More precisely, if $w_1=s_1\cdots s_m$ and $w_2=t_1\cdots t_m$, where $s_m=t_m=0$, $w_1$ and $w_2$ both contain exactly one radix point in the same position $k$, we make the make the following definitions for $i\in \{1,\ldots,k-1\}\cup \{k+1,\ldots,m\}$.
 \begin{align*}
    c_0 & =0 \\
    u_i & =s_i + t_i + c_{i-1} \mod p \\
    c_i & = 1 \ \text{ if } \ s_i + t_i + c_{i-1} \ge p \\
    c_i & = 0 \ \text{ if } \ s_i + t_i + c_{i-1}< p \\
    c_k & = c_{k-1}
 \end{align*}
We say that $u_1\cdots u_{k-1} . u_{k+1}\cdots u_m$ is the base $p$ addition with carries of $w_1$ and $w_2$.

\begin{lemma}
\label{lem: DFAO multiplication}
There is an algorithm $\mult$ which takes as input two well-ordered DFAOs $\X$ and $\Y$ both having input alphabet $\Sigma_p$ and output set $\F_q$, and outputs a well-ordered DFAO $\W$ with the same input alphabet and output set such that $\Pow(\W)=\Pow(\X)\Pow(\Y)$.
\end{lemma}

\begin{proof}
Suppose that $\mult$ is defined on minimal well-ordered DFAs. Denote this restriction of $\mult$ by $\multdfa$. Let
$$\lincomb(\X)=\{(a_i,\X_i) \ | \ i\in \{1,\ldots, 1-q\}\}$$ 
and let 
$$\lincomb(\Y)=\{(a_i,\Y_i) \ | \ i\in \{1,\ldots, 1-q\}\}.$$
Then 
$$\Pow(\X)\Pow(\Y) = \sum_{i=1}^{q-1}\sum_{j=1}^{q-1} a_ia_j\Pow(\X_i)\Pow(\Y_j).$$
Thus, $\mult$ can be defined in the following way.
\begin{algorithm}[H]
\caption*{$\mult(\X,\Y)$}
\label{alg:mult}
\begin{algorithmic}[1]
\State $\{(a_i,\X_i) \ | \ i\in \{1,\ldots, 1-q\} \gets \lincomb(\X)$
\State $\{(a_i,\Y_i) \ | \ i\in \{1,\ldots, 1-q\} \gets \lincomb(\Y)$
\State $\W\gets a_1a_1\multdfa(\X_1,\Y_1)$
\For{$i \in \{2,\ldots,q-1\}$}
    \For{$j\in \{2,\ldots,q-1\}$}
        \State $\W \gets \add(\W,a_ia_j\multdfa(\X_i,\Y_j))$
    \EndFor
\EndFor
\State \Return $\W$
\end{algorithmic}
\end{algorithm}

We will now define $\multdfa$. Suppose that $\X$ and $\Y$ are minimal well-ordered DFAs. Let $x=\Pow(\X)$ and let $y=\Pow(\Y)$. By definition of multiplication in Hahn fields, the coefficient of a term with value $\gamma$ in $xy$ is equal to the number of ways to write $r_1+r_2=\gamma$ with $r_1\in \supp(x)$ and $r_2\in \supp(y)$. Since these coefficients are in characteristic $p$, we can take this number modulo $p$. We will see that this amounts to counting the number of accepting paths modulo $p$ in a certain NFA and then use Lemma \ref{lem: counting paths} to conclude the result. To this end, let $M_x=\REV(\X)$ and let $M_y=\REV(\Y)$. Write $M_x=(Q_x,\Sigma_p,\delta_x,q_{0,x},\F_q,\tau_x)$ and $M_y=(Q_x,\Sigma_p,\delta_x,q_{0,x},\F_q,\tau_x)$. Let $S$ be the subset of $\Sigma_p^*\times \Sigma_p^*$ consisting of pairs $(w_1,w_2)$ with the following properties.
\begin{enumerate}[(1)]
    \item $w_1$ and $w_2$ have the same length.
    \item $w_1$ and $w_2$ each end with $0$.
    \item $w_1$ and $w_2$ each have a single radix point, and both are in the same position.
    \item After removing leading and trailing zeroes, $w_1$ and $w_2$ become the reversed valid base $p$ expansions of some $i\in \supp(x)$ and $j\in \supp(y)$ respectively.
\end{enumerate}

We will use Theorem \ref{thm: myhill-nerode} to show that $S$ is a regular language over $\Sigma=\Sigma_p\times\Sigma_p$, under the identification of $\Sigma^*$ with the subset of $\Sigma_p^*\times\Sigma_p^*$ where the coordinates have equal length.

To start, we will define an equivalence relation $\sim_x$ on $\Sigma^*$ as follows. Let $q\in Q_x$ and denote by $E^{(1,x)}_q$ the set of strings $w=u s_1\cdots s_kv\in\Sigma^*$ for which the following hold.
\begin{enumerate}[(i)]
    \item $u$ and $v$ are strings of only zeroes;
    \item $0\notin \{s_1, s_k\}$;
    \item $\delta^*(q_{0,x},s_1\cdots s_k v)=q$;
    \item $s_1\cdots s_k$ is accepted by $M_x$.
\end{enumerate}
Let $E^{(2,x)}_q$ be the set of strings defined as $E^{(1)}_q$, but with the fourth item replaced with
\begin{enumerate}[(iv')]
    \item $s_1\cdots s_k$ is not accepted by $M_x$.
\end{enumerate}
We have that the set $\{E^{(i,x)}_q \ | \ i\in\{1,2\}, \ q\in Q_x\}$ partitions $\Sigma^*$. Denote the equivalence relation given by this partitioning by $\sim_x$. By (iii), we have that $\sim_x$ is right-invariant. Since $Q_x$ is finite, we have that $\sim_x$ is of finite index. We define $\sim_y$ in exactly the same way, replacing all instances of $x$ with $y$.

Using $\sim_x$ and $\sim_y$, we now define an equivalence relation $\sim$ on $\Sigma^*$ as follows. Let $w=(w_1,w_2)\in\Sigma^*$. If $w_1$ and $w_2$ both contain exactly one radix point in the same position, we say that $w$ is of type A. If neither $w_1$ nor $w_2$ contains a radix point, say that $w$ is of type B. If $w$ is neither of type A nor of type B, we say that $w$ is of type C. Let $w'=(w'_1,w'_2)\in\Sigma^*$. We define $w\sim w'$ to hold if $w$ and $w'$ are both of type C, or if the following hold.
\begin{enumerate}[(a)]
    \item $w$ is of the same type as $w'$;
    \item $w_1\sim_xw'_1$;
    \item $w_2\sim_yw'_2$.
\end{enumerate}
Since $\sim_x$ and $\sim_y$ are right-invariant and of finite index, so is $\sim$. It is immediate that $S$ is the union of equivalence classes of $\Sigma^*/\sim$, namely the equivalence classes of elements $w=(w_1,w_2)$ of type A with $w_1\in E_q^{(1,x)}$ and $w_2\in  E_{q'}^{(1,y)}$ for some $q\in Q_x$ and some $q'\in Q_y$. We conclude that $S$ is regular.

We will now construct an explicit DFA $M=(Q,\Sigma,\delta,q_0,F)$ which accepts $S$. This is essentially an application of Corollary 4.1.9 in \cite{MR1997038}. Let $Q$ be a set of representatives of $\Sigma^*/\sim$. Such a set can be obtained by considering strings of bounded length, as in the definition of $\relevantstates$. For $w\in Q$ and $s\in \Sigma$, let $\delta(w,s)=w'$, where $ws \sim w'$. Let $q_0$ be the element in $Q$ which is equivalent to the empty string under $\sim$. Let $F=Q\cap S$. By construction, we have that $M$ accepts $S$.

Informally, the DFA $M$ can be seen to accept the pairs of reversed base $p$ expansions of the exponents in $\supp(x)\times \supp(y)$ which are well set up for adding these pairs together under base $p$ addition with carries. We will now construct an NFA $M' = (Q', \Sigma_p, \delta', q'_0, F')$ which captures this addition. Let $Q' = Q \times \{0, 1\}$, let $q'_0 = (q_0, 0)$ and let $F' = F \times \{0\}$, where $F$ is the set of accepting states of $M$. To define $\delta'$, let $(q,i)\in Q'$ and consider the following cases.
\begin{enumerate}
    \item If $s \in \{0, \ldots , p-1\}$, we include $(q',0)$ (resp. $(q',1)$) in $\delta'((q,i),s)$ if there exists a pair $(t, u) \in \{0,\ldots, p-1\}^2$ with $t + u + i < p$ (resp. $t + u + i \ge p$) and $t + u + i \equiv s \mod p$ such that $\delta(q,(t, u)) = q'$.
    \item If $s$ is equal to the radix point, we include $(q', i)$ in $\delta((q, i), s)$ if $\delta(q,(s, s)) = q'$, and we never include $(q', 1 - i)$. 
\end{enumerate}

It is shown in the proof of \cite[Lemma 2.2.2]{MR2289431} that the number of accepting paths of $w\in\Sigma_p^*$ in $M'$ is equal to the number of pairs $(w_1,w_2)\in S$ which sum to $w$ with its leading and trailing zeroes under ordinary base $p$ addition with carries. Let $N=\accpaths(M')$, as in Lemma \ref{lem: counting paths}. So the function $f_N(w)$ equals the number of accepting paths of $w$ in $M'$ modulo $p$.

A priori, there might be $(w_1,w_2)\in S$ which sum to a reversed base $p$ expansion of $\gamma$ even though $(w_1,w_2)$ does not sum to $w$, since we need to take into account leading and trailing zeroes. Appending leading zeroes to $w$ will account for more possible $(w_1,w_2)\in S$. From \cite{MR2289431}, we get a bound to how many zeroes we need to append. To make this more precise, let $m$ be greater than the number of states of $M$ and suppose that $w$ begins with $m$ leading zeroes. In the proof of \cite[Lemma 2.2.2]{MR2289431} it is shown that $0w$ then has the same number of accepting paths as $w$. 

We conclude that the function that, given the reversed valid base $p$ expansion of a number $\gamma\in\frac{1}{p^\infty}\N$, computes the mod $p$ reduction of the number of ways to write $\gamma=i+j$ with $(i,j)\in \supp(x)\times \supp(y)$ is given by a DFAO $N'=(\tilde{Q}',\Sigma_p,\tilde{\delta}',\tilde{q}'_0,\F_p, \tilde{\tau}')$ which is constructed using $N=(\tilde{Q},\Sigma_p,\tilde{\delta},\tilde{q}_0,\F_p, \tilde{\tau})$ in the following way. Let $G=(V,E)$ be the transition graph of $N$. We set the initial state of $N'$ to be $\tilde{q}'_0=\tilde{\delta}^*(\tilde{q}_0,m\cdot 0)$, where $m\cdot 0$ denotes the string with $m$ zeroes. We let the transition graph of $N'$ be $G_{\tilde{q}'_0}$. For any state $q$ of $N'$, we set $\tilde{\tau}'(q)=\tilde{\tau}(\tilde{\delta}(q,0))$. In other words, $N'$ amounts to appending $m$ leading zeroes and one trailing zero to a string and then running the result through $N$. Finally, having $\multdfa$ returning $\W=\REV(N')$ on input $(\X,\Y)$ gives the desired result.
\end{proof}

\begin{remark}
    For any $i\in\N$, let $M_i$ be the DFAO illustrated after Definition \ref{def: transition graph}, such that the accepted string $s_1\cdots s_n$ is a valid base $p$ expansion of $i$. By construction, we then have that $\Pow(M_i)=t^i$. By Lemma \ref{lem: DFAO multiplication}, this implies that for a DFAO $M$, we have $\Pow(\mult(M_i, M))=t^i\Pow(M)$. We can thus extend the notation $aM$ for $a\in \F_p$ to allow for $a=\sum_{j=0}^sa_j t^j\in \F_p[t]$, by letting $aM$ be the DFAO defined by adding together the DFAOs $a_jM_j$ using the algorithm $\add$.
\end{remark}

\begin{remark}
There is an algorithm $\power^i$ for any $i\in \N$ which takes as input a well-ordered DFAO $M$ and returns a DFAO $N$ such that $\Pow(N)=\Pow(M)^i$. The sequence of such algorithms $(\power^i)_{i\in \N}$ is defined recursively by letting $\power^0(M)=\mathbbm{1}$ and
$$\power^{i+1}(M)=\mult(M,\power^i(M)).$$
\end{remark}

\begin{remark}
\label{rem: is root}
There is an algorithm $\isrootalg$ which takes as input a polynomial $f(X)\in\F_p[t][X]$ and a DFAO $\X$ and returns $\true$ if $f(\Pow(\X))=0$ and $\false$ otherwise. It is defined as follows.
\begin{algorithm}[H]
\caption*{$\isrootalg(f=\sum_{i=0}^na_iX^i,\X)$}
\label{alg:isroot}
\begin{algorithmic}
\State $\W \gets a_0\mathbbm{1}$ 
\For{$i\in\{1,\ldots,n\}$}
    \State $\X^i\gets \power^i(\X)$
    \State $\W\gets \add(\W,a_i\X^i)$
\EndFor
\If{$\relevantstates(\W) = \emptyset$}
    \State \Return $\true$
\EndIf
\State \Return $\false$
\end{algorithmic}
\end{algorithm}
\end{remark}

\end{document}